\documentclass[a4paper,12pt,reqno]{amsart}
\usepackage{amsmath,amsfonts,amssymb,amsthm,enumerate}
\usepackage{hyperref}
\usepackage[utf8]{inputenc}

\newtheorem{theorem}{Theorem}[section]

\newtheorem{example}{Example}[section]
\newtheorem{corollary}{Corollary}[section]
\newtheorem{remark}{Remark}[section]

\newcommand{\cA}{\mathcal{A}}
\newcommand{\cB}{\mathcal{B}}
\newcommand{\cG}{\mathcal{G}}
\newcommand{\cH}{\mathcal{H}}
\newcommand{\cU}{\mathcal{U}}
\newcommand{\cV}{\mathcal{V}}
\newcommand{\fraka}{\mathfrak{a}}
\newcommand{\frakb}{\mathfrak{b}}
\newcommand{\bfa}{\mathbf{a}}
\newcommand{\bfb}{\mathbf{b}}
\newcommand{\bfe}{\mathbf{e}}
\newcommand{\bff}{\mathbf{f}}
\newcommand{\bfm}{\mathbf{m}}
\newcommand{\bfn}{\mathbf{n}}

\numberwithin{equation}{section}

\title[Digamma series]{Evaluation of terminating and non-terminating sums containing the digamma function}

\author{Sergei Kalmykov$^{1,2}$}
\address{$^{1}$School of Mathematical Sciences, CMA-Shanghai, Shanghai Jiao Tong University, 800 Dongchuan RD, Shanghai 200240, China} 
\address{$^{2}$
Keldysh Institute of Applied Mathematics 
of Russian Academy of Sciences,
Miusskaya pl., 4,
125047, Moscow,
Russia}
\email{kalmykovsergei@sjtu.edu.cn} 

\author{Dmitrii Karp$^{3,4}$}
\address{$^{3}$Department of Mathematics, Holon Institute of Technology, Holon, Israel} 
\address{$^{4}$Institute of Mathematics and Informatics, Bulgarian Academy of Sciences, Sofia, Bulgaria} 
\email{dimkrp@gmail.com} 

\author{Vinay Shukla$^{5,6,\#}$}{\thanks{$^{\#}$Corresponding author}}
\address{$^{5}$School of Mathematical Sciences,  Shanghai Jiao Tong University, 800 Dongchuan RD, Shanghai 200240, China} 
\address{$^{6}$Department of Mathematics, School of Computer Science Engineering and Technology, Bennett University, Greater Noida, India} 
\email{vinayshukla4321@gmail.com, vnshukla01@sjtu.edu.cn}

\allowdisplaybreaks

\begin{document}
\keywords{Digamma series, Hypergeometric series, summation formula, hypergeometric identity, degeneration process}
\subjclass[2020] {33C20, 33C05, 33B15}
\begin{abstract}

We derive transformation and summation formulas for terminating and
nonterminating series involving the digamma function.  Our principal results are obtained by a limiting process starting with duality relations for the generalized hypergeometric functions and their consequences. Selected formulas are further extended by parameter differentiation of Euler's transformation and by using
contiguous relations. 
Most of our identities express products of
hypergeometric and digamma  series in terms of hypergeometric functions, and some evaluations of terminating digamma sums involve Bernoulli polynomials.  In several cases the digamma contributions
cancel, producing identities involving only products of hypergeometric
functions. 
\end{abstract}
\maketitle

\section{Introduction}

Series involving the digamma function $\psi(z)=\Gamma'(z)/\Gamma(z)$ arise naturally in several contexts. The first one appears when a generalized hypergeometric series is differentiated with respect to parameters, in view of 
$$
\frac{d}{da}(a)_n = (a)_n [\psi(a+n)-\psi(a)],\qquad\frac{d}{da}\frac{1}{(a)_n}= \frac{1}{(a)_n}[\psi(a)-\psi(a+n)],
$$
where $(a)_{n}=\Gamma(a+n)/\Gamma(a)$ stands for the Pochhammer symbol (also called the rising factorial).  Another context, where digamma series play a prominent role, is evaluation of series and sums of harmonic numbers, in particular, those that arise in the computation of $\pi$; see, for instance, \cite{Chu-Campbell}.

Generally speaking a series containing digamma functions is not hypergeometric (the ratio of consecutive terms is not a rational function of the summation index), but frequently can be written in terms of multivariate hypergeometric functions, see, for instance \cite{Ancarani_Gasaneo_2008,Ancarani_Gasaneo_2009,Ancarani_Gasaneo_2010,Apelblat_2020,Brychkov_Geddes_2004,Cvijovi_Miller_2010}.  In special situations, however, such series  can be summed in terms of  gamma and digamma functions or transformed into univariate hypergeometric series.  One reason such transformations may be useful is numerical evaluation which may be delicate for digamma series while for hypergeometric series various well-established numerical routines are readily available. Furthermore, numerous identities and transformation formulas known for hypergeometric functions may furnish additional insights into the original digamma series.  

The first comprehensive collection of evaluations of such $\psi$-series is likely due to Hansen \cite[Section 55, pp. 360–366]{hansen1975table}; additional formulas, numbered as $55.*$, appear in the addendum by Borwein \cite{borweinaddendum}. A systematic summary of results up to 2008 is provided in the monograph by Brychkov \cite{Brychkov_2008}, especially in Sections 1.30.2 and 6.2. Miller's 2006 paper \cite{Miller_2006} is the first modern paper in the line that treats digamma-weighted hypergeometric-type series as objects that can sometimes be reduced to closed hypergeometric/special-function form. Cvijović \cite{Cvijovi_2008} explicitly positions itself as a refinement and extension of Miller’s work. González-Santander and Sánchez Lasheras use exactly this method in \cite{Santander_Juan_2022} and \cite{Santander_Juan_2023}, explicitly stating that they derive finite and infinite digamma sums by differentiating reduction formulas.  Several works on digamma series are motivated by physical applications, particularly evaluation of Feynman diagrams, see \cite{Ancarani_Gasaneo_2010,Bera_2023,Coffey_2005,greynat2014new,kalmykov2010all}.

The present paper studies identities obtained by applying a limit process to multi-term hypergeometric identities from \cite{Cetinkaya_Karp_2021,Karp_Kuznetsov_2021}. Such a limiting process is needed when some terms in the original identities become singular, but their combinations can still be regularized and attain finite values.  This process leads to summation and transformation formulas for infinite series and finite sums involving the digamma function. In some cases the digamma terms cancel completely, leading instead to presumably new multi-term hypergeometric identities. Motivated by formulas obtained by this limit process, we generalize some of them  by manipulating parameter derivatives of known hypergeometric identities. A similar approach has been previously applied  by A.\:\c{C}etinkaya and the second author in \cite{Asena_Karp_2025} to compute degenerate forms of multi-term identities for the generalized hypergeometric series evaluated at unity.  In contrast, in this paper, we derive identities containing an arbitrary argument (properly restricted to ensure convergence).  To this end we start with duality relations involving sums of products of hypergeometric series and apply degeneration to them.

This paper is structured as follows. Section~\ref{sec:duality-digamma} applies a degeneration process to the duality relation in \cite[Theorem~1]{Karp_Kuznetsov_2021}, leading to sum-of-product identities with one factor being digamma series while the other being hypergeometric.  This entails several corollaries including summation formulas, which we generalize using a different method.

Section \ref{sec:finite-digamma-sums} applies the limiting process to Theorem~6.2 of \cite{Cetinkaya_Karp_2021}, thereby obtaining a novel summation formula for a terminating digamma sum by expressing it in terms of hypergeometric series and Bernoulli polynomials. Section~\ref{sec:contiguous-degeneration} examines limiting cases of the identities in \cite[Lemmas~6.4 and~6.5]{Cetinkaya_Karp_2021}, leading somewhat surprisingly to purely hypergeometric identities.  These identities are then extended by a technique based on contiguous relations. 
For better readability, proofs obtained by degeneration are
collected in the Appendix whenever the corresponding result is later
generalized by a different method.

\section{From duality relations to digamma series}\label{sec:duality-digamma}

Let us fix the notation. Given an integer $r\ge1$, denote
\begin{equation}\label{eq:integer-vector-notation}
\begin{split}
&\fraka=(\alpha_1,\ldots,\alpha_r)\in\mathbb{C}^r,\quad \bfn=(n_1,\ldots,n_r)\in\mathbb{Z}^r,\quad
\bfm=(m_1,\ldots,m_r)\in\mathbb{Z}^r,
\\[6pt]
&M=m_1+\cdots+m_r,\qquad
N=n_1+\cdots+n_r,\qquad
e=\max\{-1,M-N-r+1\},\\[6pt]
& m_{\min} = \min_{1 \leq i \leq r} m_i, \hspace{2cm} n_{\max} = \max_{1 \leq i \leq r} n_i.
\end{split}
\end{equation}
We will further use the abbreviations 
\begin{align*}
&\Gamma(\fraka)=\prod_{i=1}^r\Gamma(\alpha_i),\qquad
(\fraka)_k=\prod_{i=1}^r(\alpha_i)_k,\qquad
(\fraka)_\bfn=\prod_{i=1}^r(\alpha_i)_{n_i},\qquad
\sin(\fraka)=\prod_{i=1}^r\sin(\alpha_i),\\
&\psi(\fraka)=\sum_{i=1}^r\psi(\alpha_i),\quad
\fraka+\gamma=(\alpha_1+\gamma,\ldots,\alpha_r+\gamma),\quad\fraka_{[k]}=(\alpha_1,\ldots,\alpha_{k-1},\alpha_{k+1},\ldots,\alpha_r),
\end{align*}
where  $\gamma$ is a scalar.   Sums and products over an empty parameter vector
are  interpreted as $0$ and $1$, respectively. 
The generalized  hypergeometric series ${}_rF_{r-1}$ and its regularized version ${}_r\phi_{r-1}$ are defined in the standard way, namely
\begin{align}
{}_r\phi_{r-1}\!\left(\!\begin{array}{c}
 \fraka\\ \frakb
\end{array}\middle|z\right)
:=\frac{\Gamma(\fraka)}{\Gamma(\frakb)}
{}_rF_{r-1}\!\left(\!\begin{array}{c}
 \fraka\\ \frakb
\end{array}\middle|z\right)
=\sum_{k=0}^{\infty}
\frac{\Gamma(\fraka+k)z^k}{\Gamma(\frakb+k)k!},
\label{eq:regularized-hypergeometric}
\end{align}
where $\frakb\in\mathbb{C}^{r-1}$. The function ${}_r\phi_{r-1}$ is also well-defined and finite when some components of $\frakb$ are non-positive integers. We will further omit the size indices whenever they can be read off from the sizes of the parameter vectors; thus $F$ and $\phi$ without indices denote the corresponding generalized hypergeometric and regularized hypergeometric functions, respectively. Products, quotients, and shifts involving parameter vectors are understood component-wise. If the argument of the above functions is omitted, it is understood to be equal to $1$:
\[
{}_rF_{r-1}\!\left(\begin{array}{c}\fraka\\ \frakb\end{array}\right)
:=
{}_rF_{r-1}\!\left(\begin{array}{c}\fraka\\ \frakb\end{array}\middle|1\right),
\qquad
{}_r\phi_{r-1}\!\left(\begin{array}{c}\fraka\\ \frakb\end{array}\right)
:=
{}_r\phi_{r-1}\!\left(\begin{array}{c}\fraka\\ \frakb\end{array}\middle|1\right).
\]

In \cite{Karp_Kuznetsov_2021}, Alexey Kuznetsov and the second author derived a duality relation for the generalized hypergeometric series which covers many previously known identities of this type. Given $\fraka\in\mathbb{C}^r$ whose components are distinct modulo integers, $\frakb\in\mathbb{C}^r$ and $|z|<1$, according to \cite[eq.(3)]{Karp_Kuznetsov_2021} we have
\begin{equation}\label{eq:duality-relation}
V_{1}(z)=\frac{1}{(1-z)^{e+1}}\sum_{j=-n_{\max}}^{e-m_{\min}} \Delta_j z^j,    
\end{equation}
with 
\begin{equation}\label{eq:v-s-definition}
V_{s}(z)= \sum_{i=s}^r  \frac{(1-\frakb+\alpha_i)_{\bfm-n_i} z^{-n_i}}{(\alpha_i-\fraka_{[i]})_{\bfn_{[i]}-n_i+1}} F \left(\begin{array}{c}
 \frakb-\alpha_i\\ 1+\fraka_{[i]}-\alpha_i
\end{array}\middle| z \right) F \left(\begin{array}{c}
 1-\frakb+\alpha_i+\bfm-n_i\\ 1-\fraka_{[i]}+\alpha_i+\bfn_{[i]}-n_i
\end{array}\middle| z \right),
\end{equation}
where the coefficients $\Delta_j$ have been computed in \cite[Lemma~6.1]{Cetinkaya_Karp_2021} as follows:
\begin{multline}\label{eq:delta-j-definition}
\Delta_j = \sum_{\ell=\max(-n_{\max},j-e-1)}^j \binom{e+1}{j-\ell}(-1)^{j-\ell}\sum_{i=1}^r \frac{(1-\frakb+\alpha_i)_{\bfm+\ell}}{(\alpha_i-\fraka_{[i]})_{\bfn_{[i]}+\ell+1} (\ell+n_i)!}\\
\times{}_{2r}F_{2r-1} \left(\begin{array}{c}
 -\ell-n_i,\frakb-\alpha_i, \fraka_{[i]}-\alpha_i-\bfn_{[i]}-\ell\\ \frakb-\alpha_i-\bfm-\ell,1+\fraka_{[i]}-\alpha_i
\end{array}\right).
\end{multline}
Note that this formula has been further generalized in a recent preprint \cite{KarpZhang2026}.
Formula \eqref{eq:duality-relation} fails if any two (or more) components of $\fraka$ differ by an integer, as some of the hypergeometric functions on the left-hand side have poles. It is, however, not difficult to see that the expression on the right-hand side remains finite, so that the singularities appearing on the left-hand side must cancel out.  In the theorem below, we compute the explicit form of the degenerate left-hand side leading to an identity expressing a sum of products of hypergeometric and digamma series as a finite sum of products of hypergeometric series and elementary functions.

\begin{theorem}\label{th:duality-degeneration}
Let $r\ge2$, fix $p\in\mathbb{Z}$, and let $\fraka,\frakb\in\mathbb{C}^r$. Assume that $\alpha_2=\alpha_1+p$ and that $\alpha_i-\alpha_j\notin\mathbb Z$ for every pair $\{i,j\}\ne\{1,2\}$. Denote $\cA=1-\fraka_{[1,2]}+\alpha_1+p$ \emph{(}having dimension $r-2$\emph{)} and $\cB=1-\frakb+\alpha_1+p$ \emph{(}having dimension $r$\emph{)} and suppose these two vectors do not contain integer components. For $0<|z|<1$, the following identity holds
\begin{align}\label{eq:duality-degeneration} 
&\phi \left(\begin{array}{c}
 \cB+\bfm-n_2\\1+p+{n}_1-n_2,~ \cA+\bfn_{[1,2]}-n_2
\end{array} \middle| z\right) \sum_{k=0}^\infty [\Psi_1(k)z^p+\Psi_2(k)]z^{k-n_2} \nonumber\\
& + \phi\left(\begin{array}{c}
 1-\cB+p\\ 1+p,~ 2-\cA+p
\end{array} \middle| z\right)\sum_{k=0}^\infty [\Psi_3(k)z^{p+n_1-n_2}-\Psi_4(k)]z^{k-n_1} \nonumber\\
&=\! (-1)^{r+p}\pi^2\frac{\sin (\pi\cA)}{\sin (\pi\cB)}\bigg(\frac{1}{(1-z)^{e+1}}\!\!\sum_{j=-n_{\max}}^{e-m_{\min}} \Delta_j z^j -V_{3}(z) \bigg)+\pi z^{-n_2}\bigg[\sum_{\eta \in \cA} \cot \pi\eta -\sum_{\delta \in \cB} \cot \pi\delta \bigg] \nonumber\\
&\times \phi\left(\begin{array}{c}
 1-\cB+p\\ 1+p,~ 2-\cA+p
\end{array} \middle| z\right) \phi \left(\begin{array}{c}
 \cB+\bfm-n_2\\1+p+{n}_1-n_2,~ \cA+\bfn_{[1,2]}-n_2
\end{array} \middle| z\right),
\end{align}
where $V_{3}(z)$ is defined in \eqref{eq:v-s-definition}, the coefficients $\Delta_j$ are the continuous limits of \eqref{eq:delta-j-definition} as \(\alpha_2\to\alpha_1+p\), and
\begin{align*}
&\Psi_1(k) =\frac{\psi(1+p+k)\Gamma(1-\cB+p+k)}{\Gamma(1+p+k) \Gamma(2-\cA+p+k)k!}, \\
&\Psi_2(k) =\frac{\Gamma(1-\cB+k)}{\Gamma(1-p+k)\Gamma(2-\cA+k)k!} \Big\{\psi(1-p+k) +\psi(2-\cA+k) - \psi(1-\cB+k)\Big\},\\
&\Psi_3(k) =\frac{\Gamma(\cB+\bfm+k-n_2)}{\Gamma(1+p+k+{n}_1-n_2)\Gamma(\cA+k+\bfn_{[1,2]}-n_2)k!} \\
& \times\Big\{\psi(\cB+\bfm+k-n_2)-\psi(1+p+n_1-n_2+k) - \psi(\cA+\bfn_{[1,2]}-n_2+k)\Big\},\\
&\Psi_4(k) = \frac{\psi(1-p+k+{n}_2-n_1)\Gamma(\cB-p+\bfm+k-n_1)}{\Gamma(1-p+k+{n}_2-n_1)\Gamma(\cA-p+k+\bfn_{[1,2]}-n_1)k!} .
\end{align*}
\end{theorem}
\begin{proof}
Suppose $\alpha_2 =\alpha_1+p+\epsilon$, $p \in \mathbb{Z}_{\geq 0}$. Separating the terms with $i=1$ and $i=2$ in  \eqref{eq:duality-relation}-\eqref{eq:v-s-definition} we can write  
\begin{equation}\label{eq:duality-split}
f_1(\epsilon)+ f_2(\epsilon) 
=\frac{1}{(1-z)^{e+1}}\sum_{j=-n_{\max}}^{e-m_{\min}} \Delta_j z^j - V_{3}(z),
\end{equation}
where after simple transformations using the reflection formula $\Gamma(z)\Gamma(1-z)=\pi/\sin(\pi z)$ and the relation $\sin \pi(z+p) = (-1)^p \sin \pi z$ and in the view of \eqref{eq:regularized-hypergeometric}, we will have 
\begin{align*}
f_1(\epsilon)= \frac{(-1)^{p+1}}{\pi\sin \pi\epsilon} h_1(\epsilon), \quad
f_2(\epsilon)= \frac{(-1)^{p} }{\pi\sin \pi\epsilon}h_2(\epsilon),
\end{align*}
where
\begin{equation}\label{eq:h-functions}
\begin{aligned}
&h_1(\epsilon) = \frac{\sin \pi(\frakb-\alpha_1)z^{-n_1}}{\sin \pi(\alpha_1-\fraka_{[1,2]})} \phi\left(\begin{array}{c}
 \frakb-\alpha_1\\ 1+p+\epsilon,~ 1+\fraka_{[1,2]}-\alpha_1
\end{array} \middle| z\right)\\
& \hspace{3cm} \times\phi \left(\begin{array}{c}
 1-\frakb+\alpha_1+\bfm-n_1\\1-p-\epsilon+{n}_2-n_1,~ 1-\fraka_{[1,2]}+\alpha_1+\bfn_{[1,2]}-n_1
\end{array} \middle| z\right),\\[10pt]
&h_2(\epsilon) = \frac{\sin \pi(\frakb-\alpha_1-p-\epsilon)z^{-n_2}}{\sin \pi(\alpha_1+p+\epsilon-\fraka_{[1,2]})} \phi\left(\begin{array}{c}
 \frakb-\alpha_1-p-\epsilon\\ 1-p-\epsilon,~ 1+\fraka_{[1,2]}-\alpha_1-p-\epsilon
\end{array} \middle| z\right) \\
&\hspace{3cm}\times\phi \left(\begin{array}{c}
 1-\frakb+\alpha_1+p+\epsilon+\bfm-n_2\\1+p+\epsilon+{n}_1-n_2,~ 1-\fraka_{[1,2]}+\alpha_1+p+\epsilon+\bfn_{[1,2]}-n_2
\end{array} \middle| z\right).
\end{aligned}
\end{equation}
Further, employing Taylor expansion for $h_1(\epsilon)$ and $h_2(\epsilon)$, we get
\begin{align}\label{eq:duality-limit}
f_1(\epsilon)+f_2(\epsilon) &= \frac{(-1)^{p} }{\pi\sin \pi\epsilon} [h_2(\epsilon)-h_1(\epsilon)] \notag\\
&= \frac{(-1)^{p}}{ \pi\sin \pi\epsilon} [h_2(0)+\epsilon h'_2(0)-h_1(0)-\epsilon h'_1(0)+O(\epsilon^2)] \notag\\
& = \frac{(-1)^{p} \pi \epsilon}{\pi^2\sin \pi\epsilon} [h'_2(0)-h'_1(0)+O(\epsilon)] \rightarrow \frac{(-1)^{p}}{\pi^2} [h'_2(0)-h'_1(0)] ~~{\rm as}~~ \epsilon \rightarrow 0,
\end{align}
where the last equality follows from the relation $h_1(0)=h_2(0)$, which we will now prove. It is enough to consider \(p\ge0\): for \(p<0\), interchange the indices \(1\) and \(2\) in the underlying duality relation and replace \(p\) by \(-p\). Suppose first that $p>0$, so that $\Gamma(1-p+k)$ has a pole for $k=0,1,\ldots,p-1$. Then, for the first hypergeometric factor in $h_2$ at $\epsilon=0$, we have
\begin{align*}
&\phi\!\left(\!\!\begin{array}{c}
 \frakb-\alpha_1-p\\ 1-p,~ 1+\fraka_{[1,2]}-\alpha_1-p
\end{array} \middle| z\right) =\sum_{k=p}^{\infty} \frac{\Gamma(\frakb-\alpha_1-p+k)z^k}{\Gamma(1-p+k)\Gamma(1+\fraka_{[1,2]}-\alpha_1-p+k)k!}  \nonumber\\
&=z^p\sum_{k=0}^{\infty} \frac{\Gamma(\frakb-\alpha_1+k)z^{k}}{\Gamma(1+p+k)\Gamma(1+\fraka_{[1,2]}-\alpha_1+k)k!} = z^p\phi\left(\begin{array}{c}
 \frakb-\alpha_1\\ 1+p,~ 1+\fraka_{[1,2]}-\alpha_1
\end{array} \middle| z\right),
\end{align*}
which aside from $z^p$ coincides with the first factor in $h_1(0)$.   
In a similar fashion, for the second hypergeometric factor in $h_1$ at $\epsilon=0$,  we obtain 
\begin{multline*}
z^{-n_1}\phi \left(\begin{array}{c}
 1-\frakb+\alpha_1+\bfm-n_1\\1-p+{n}_2-n_1,~ 1-\fraka_{[1,2]}+\alpha_1+\bfn_{[1,2]}-n_1
\end{array} \middle| z\right)
\\=z^{-n_1}\sum_{k=p+n_1-n_2}^{\infty} \frac{\Gamma(1-\frakb+\alpha_1+\bfm-n_1+k)z^k}{\Gamma(1-p+{n}_2-n_1+k)\Gamma(1-\fraka_{[1,2]}+\alpha_1+\bfn_{[1,2]}-n_1+k)k!}
\\=z^p\sum_{k=0}^{\infty} \frac{\Gamma(1-\frakb+\alpha_1+\bfm+k+p-n_2)z^{k-n_2}}{\Gamma(1+k+p+n_1-n_2)\Gamma(1-\fraka_{[1,2]}+\alpha_1+\bfn_{[1,2]}+k+p-n_2)k!}\\
=z^{p-n_2} \phi \left(\begin{array}{c}
 1-\frakb+\alpha_1+p+\bfm-n_2\\1+p+{n}_1-n_2,~ 1-\fraka_{[1,2]}+\alpha_1+p+\bfn_{[1,2]}-n_2
\end{array} \middle| z\right),
\end{multline*}
which aside from a power of $z$ coincides with the second factor in $h_2(0)$. 
The index shift is valid for either sign of \(p+n_1-n_2\): if this integer is negative, the formally added negative-index terms vanish because \(1/\Gamma(k+1)=0\) for negative integers \(k\).
Finally, it is clear that
\begin{align*}
\frac{\sin \pi(\frakb-\alpha_1-p)}{\sin \pi(\alpha_1+p-\fraka_{[1,2]})} = \frac{\sin \pi(\frakb-\alpha_1)}{\sin \pi(\alpha_1-\fraka_{[1,2]})}=
(-1)^r\frac{\sin(\pi\cB)}{\sin(\pi\cA)}.
\end{align*}
Thus, setting $\epsilon=0$ in \eqref{eq:h-functions} we conclude that $h_1(0)=h_2(0)$. Next, to compute $h'_1(0)$, we will use the following relation:
\begin{align}\label{eq:reciprocal-gamma-derivative}
\psi(z) = \frac{\Gamma'(z)}{\Gamma(z)} \Rightarrow \frac{d}{d\epsilon}\frac{1}{\Gamma(z\pm \epsilon)} = \mp \frac{\psi (z\pm \epsilon)}{\Gamma(z\pm \epsilon)}.
\end{align}
It implies that
\begin{align*}
&h'_1(0) = \left[ \frac{\partial}{\partial\epsilon} h_1(\epsilon)  \right]_{\epsilon = 0}=-\frac{\sin \pi(\frakb-\alpha_1)z^{-n_1}}{\sin \pi(\alpha_1-\fraka_{[1,2]})} \bigg[ \sum_{k=0}^\infty \frac{\psi(1+p+k)\Gamma(\frakb-\alpha_1+k)z^k}{\Gamma(1+p+k) \Gamma(1+\fraka_{[1,2]}-\alpha_1+k)k!}  \\
& \times \phi\! \left(\!\!\begin{array}{c}
 1-\frakb+\alpha_1+\bfm-n_1\\1-p+{n}_2-n_1,~ 1-\fraka_{[1,2]}+\alpha_1+\bfn_{[1,2]}-n_1
\end{array} \middle| z\right) - \phi\!\left(\!\!\begin{array}{c}
 \frakb-\alpha_1\\ 1+p,~ 1+\fraka_{[1,2]}-\alpha_1
\end{array}\middle| z \right)  \\
&\times \sum_{k=0}^\infty \frac{\psi(1-p+{n}_2-n_1+k)\Gamma(1-\frakb+\alpha_1+\bfm-n_1+k)z^k}{k!\Gamma(1-p+{n}_2-n_1+k)\Gamma(1-\fraka_{[1,2]}+\alpha_1+\bfn_{[1,2]}-n_1+k)} \bigg].
\end{align*}
To compute $h'_2(0)$, we will use the following differentiation rules:
\begin{align*}
& \frac{\partial}{\partial\epsilon} \frac{\Gamma (\mathfrak{c}+k\pm \epsilon)}{\Gamma (\mathfrak{d}+k\pm \epsilon)} =\pm \frac{\Gamma (\mathfrak{c}+k\pm \epsilon)}{\Gamma (\mathfrak{d}+k\pm \epsilon)} \left[ \sum_{c \in \mathfrak{c}} \psi (c+k\pm \epsilon) - \sum_{d \in \mathfrak{d}} \psi (d+k\pm \epsilon)\right],\\
&\frac{\partial}{\partial\epsilon} \frac{\sin\pi(\mathfrak c-\epsilon)}
     {\sin\pi(\mathfrak d+\epsilon)}
=
-\pi\frac{\sin\pi(\mathfrak c-\epsilon)}
          {\sin\pi(\mathfrak d+\epsilon)}
\left[
 \sum_{c\in\mathfrak c}\cot\pi(c-\epsilon)
 +\sum_{d\in\mathfrak d}\cot\pi(d+\epsilon)
\right].
\end{align*}
Recalling the definitions $\cA=1-\fraka_{[1,2]}+\alpha_1+p$  and $\cB=1-\frakb+\alpha_1+p$, we then arrive at
\begin{align*}
&h'_2(0)\! =\! \left[ \frac{\partial}{\partial\epsilon} h_2(\epsilon)  \right]_{\epsilon=0}\!\!=-\frac{\pi\sin \pi(\frakb-\alpha_1)z^{-n_2}}{\sin \pi(\alpha_1-\fraka_{[1,2]})} \bigg[\bigg\{\sum_{\beta \in \frakb} \cot \pi(\beta- \alpha_1)-\sum_{\alpha \in \fraka_{[1,2]}} \cot \pi(\alpha - \alpha_1) \bigg\}\\
&\times \phi\!\left(\!\!\begin{array}{c}
 1-\cB\\ 1-p,~ 2-\cA
\end{array} \middle| z\right) \phi\! \left(\!\!\begin{array}{c}
\cB+\bfm-n_2\\1+p+{n}_1-n_2,\cA+\bfn_{[1,2]}-n_2
\end{array} \middle| z\right)\\[6pt]
&\hspace{2cm}-\frac{1}{\pi}\phi\left(\!\!\begin{array}{c}
 \cB+\bfm-n_2\\1+p+{n}_1-n_2,\cA+\bfn_{[1,2]}-n_2
\end{array}\middle| z \right) \\
&\times\sum_{k=0}^\infty \frac{\Gamma(1-\cB+k)z^k}{k!\Gamma(1-p+k)\Gamma(2-\cA+k)}\{\psi(1-p+k) +\psi(2-\cA+k) -  \psi(1-\cB+k)\}\\[6pt]
& -\frac{1}{\pi} \phi\left(\!\!\begin{array}{c}
 1-\cB\\ 1-p, 2-\cA
\end{array} \middle| z\right)\sum_{k=0}^\infty \frac{z^k}{k!}\frac{\Gamma(\cB+\bfm-n_2+k)}{\Gamma(1+p+{n}_1-n_2+k)\Gamma(\cA+\bfn_{[1,2]}-n_2+k)}  \\
& \times\big\{\psi(\cB+\bfm-n_2+k)-\psi(1+p+{n}_1-n_2+k)- \psi(\cA+\bfn_{[1,2]}-n_2+k)\big\}\bigg]. 
\end{align*}

Now, formula \eqref{eq:duality-degeneration} is the limit of \eqref{eq:duality-split} as $\epsilon\to0$ computed using \eqref{eq:duality-limit} in view of the above expressions for $h'_2(0)$, $h'_1(0)$ and the elementary relations $\cot \pi(x \pm p)= \cot \pi x$ and $\cot \pi(1-x)=-\cot \pi x$.
\end{proof}

\begin{remark}
It is worth noting that the ratio $\psi(y)/\Gamma(y)$ appearing in the expressions for $\Psi_j(k)$ and in the course of the proof is well defined for negative integer values of  $y$. Indeed, 
\begin{equation*}
\frac{d}{dz} \frac{1}{\Gamma(z)} = \frac{d}{dz} \frac{\sin (\pi z) \Gamma(1-z)}{\pi} = \cos (\pi z)\Gamma(1-z)-\frac{\sin (\pi z) \psi(1-z) \Gamma(1-z)}{\pi} = -\frac{\psi(z)}{\Gamma(z)},
\end{equation*}
so that 
\begin{align}\label{eq:psi-over-gamma-negative}
 \frac{\psi(-n)}{\Gamma(-n)} =   -\cos  (-\pi n)\Gamma(1+n)= (-1)^{n+1} n!.
\end{align}
\end{remark}

If  $m_{\min} > e+n_{\max}$, then the first sum on the right-hand side of \eqref{eq:duality-degeneration} vanishes and we get
\begin{corollary}
Under conditions of Theorem~\ref{th:duality-degeneration} and assuming further that
$m_{\min} > e+n_{\max}$, the following identity holds:
\begin{align}\label{eq:duality-no-delta}
&\phi \left(\begin{array}{c}
 \cB+\bfm-n_2\\1+p+{n}_1-n_2,~ \cA+\bfn_{[1,2]}-n_2
\end{array} \middle| z\right) \sum_{k=0}^\infty [\Psi_1 (k)z^p+\Psi_2(k)]z^{k-n_2} \notag\\
&+ \phi\left(\begin{array}{c}
 1-\cB+p\\ 1+p,~ 2-\cA+p
\end{array}  \middle| z \right)\sum_{k=0}^\infty [\Psi_3(k)z^{p+n_1-n_2}-\Psi_4(k)]z^{k-n_1} \notag\\
&=(-1)^{r+p+1}\pi^2\:\frac{\sin (\pi\cA)}{\sin (\pi\cB)}\:V_3(z)+\pi z^{-n_2}\bigg\{\sum_{\eta \in \cA} \cot \pi\eta -\sum_{\delta \in \cB} \cot \pi\delta \bigg\}  \nonumber\\
&\times \phi\left(\begin{array}{c}
 1-\cB+p\\ 1+p,~ 2-\cA+p
\end{array}\middle| z \right)\phi \left(\begin{array}{c}
 \cB+\bfm-n_2\\1+p+{n}_1-n_2,~ \cA+\bfn_{[1,2]}-n_2
\end{array} \middle| z \right).
\end{align}
\end{corollary}

The next corollary is a summation formula for a special combination of digamma series.

\begin{corollary}\label{cr:psi-summation}
Suppose $b_1,b_2$ are real and the integers $p \geq 0$, $m_1$, $m_2$ \emph{(}with $m_1+m_2<0$\emph{)} satisfy the inequality $b_{1}+b_{2}+m_1+m_2-1<p<b_{1}+b_{2}-1$. Then the following summation formula holds
\begin{align*}
&\cG_1(b_{1},b_{2}) \sum_{k=0}^\infty \Lambda_{k}(b_{1},b_{2})+\cG_2(b_{1},b_{2})\sum_{k=0}^\infty \hat{\Lambda}_k(b_{1},b_{2})
\\
&=\frac{(-1)^p\pi^2}{\sin(\pi b_1) \sin(\pi b_2)} \sum_{j=0}^{-m_{\min}-1}\Delta_j -\pi (\cot \pi b_{1}+\cot \pi b_{2}) \cG_1(b_{1},b_{2}) \cG_2(b_{1},b_{2})
\\
&+\cG_2(b_{1},b_{2})\sum_{k=0}^{p-1} \frac{(-1)^{p-k}(p-k-1)!\Gamma(\cB-p+\bfm+k)}{k!}\\
&-\cG_1(b_{1},b_{2})\sum_{k=0}^{p-1} \frac{(-1)^{p-k}(p-k-1)!\Gamma(1-\cB+k)}{k!},
\end{align*}
where, with $\cB=(b_1,b_2)$ as before,
\begin{align*}
&\cG_1(b_{1},b_{2}) = \frac{\Gamma(b_{1}+m_1)\Gamma(b_{2}+m_2)\Gamma(1+p-b_{1}-m_1-b_{2}-m_2)}{\Gamma(1+p-b_{1}-m_1)\Gamma(1+p-b_{2}-m_2)}, \\
& \cG_2(b_{1},b_{2})=\frac{\Gamma(1-b_{1}+p)\Gamma(1-b_{2}+p)\Gamma(b_{1}+b_{2}-p-1)}{\Gamma(b_{1})\Gamma(b_{2})},\\
&\Lambda_k(b_{1},b_{2})= \frac{\Gamma(1-\cB+p+k) \big\{ \psi(1+k) - \psi(1-\cB+p+k)+ \psi(1+p+k) \big\}}{k!\Gamma(1+p+k)},\\
&\hat{\Lambda}_k(b_{1},b_{2})=\frac{\Gamma(\cB+\bfm+k) \big\{ \psi(\cB+\bfm+k)-\psi(1+p+k)-\psi(1+k) \big\}}{k!\Gamma(1+p+k)}.
\end{align*}
\end{corollary}
\begin{proof}
Substituting $z=1$, $\bfn=0$ and $r=2$ in \eqref{eq:duality-degeneration}, we obtain
\begin{align*}
&{}_2\phi_{1}\!\left(\!\!\begin{array}{c}
 \cB+\bfm\\1+p
\end{array}\!\!\right) \sum_{k=0}^\infty \frac{\Gamma(1-\cB+k)}{k!\Gamma(1-p+k)} \big\{\psi(1-p+k)-\psi(1-\cB+k)\big\}\\
& +{}_2\phi_{1}\!\left(\!\!\begin{array}{c}
 1-\cB+p\\ 1+p
\end{array}\!\!\right)\sum_{k=0}^\infty \frac{\Gamma(\cB+\bfm+k)}{k!\Gamma(1+p+k)} \big\{ \psi(\cB+\bfm+k)-\psi(1+p+k) \big\}\\
& - {}_2\phi_{1}\!\left(\!\!\begin{array}{c}
 1-\cB+p\\ 1+p
\end{array}\!\!\right) \sum_{k=0}^\infty \frac{\psi(1-p+k)\Gamma(\cB-p+\bfm+k)}{k!\Gamma(1-p+k)}  \\
&+ {}_2\phi_{1}\!\left(\!\!\begin{array}{c}
 \cB+\bfm\\1+p
\end{array}\!\!\right) \sum_{k=0}^\infty \frac{\psi(1+p+k)\Gamma(1-\cB+p+k)}{\Gamma(1+p+k)k!} \\
&=\frac{(-1)^p\pi^2}{\sin(\pi b_1) \sin(\pi b_2)}\sum_{j=0}^{-m_{\min}-1}\Delta_j-\pi \sum_{\delta \in \cB} \cot(\pi\delta) {}_2\phi_{1}\!\left(\!\!\begin{array}{c}
 1-\cB+p\\ 1+p
\end{array}\!\right) {}_2\phi_{1}\! \left(\!\!\begin{array}{c}
 \cB+\bfm\\1+p
\end{array}\!\!\right).
\end{align*}
Next, we employ \eqref{eq:psi-over-gamma-negative} in the first and third series where the terms $\psi(1-p+k)/\Gamma(1-p+k)$ are encountered, i.e., we replace
\begin{align*}
&\sum_{k=0}^\infty \frac{\psi(1-p+k)\Gamma(\cB-p+\bfm+k)}{k!\Gamma(1-p+k)}=\sum_{k=0}^{p-1} \frac{\psi(1-p+k)\Gamma(\cB-p+\bfm+k)}{k!\Gamma(1-p+k)}\\
&+\sum_{k=p}^\infty \frac{\psi(1-p+k)\Gamma(\cB-p+\bfm+k)}{k!\Gamma(1-p+k)}= \sum_{k=0}^{p-1} \frac{(-1)^{p-k}(p-k-1)!\Gamma(\cB-p+\bfm+k)}{k!}\\
&+\sum_{k=0}^\infty \frac{\psi(1+k)\Gamma(\cB+\bfm+k)}{k!\Gamma(1+p+k)},
\end{align*}
and 
\begin{align*}
&\sum_{k=0}^\infty \frac{\Gamma(1-\cB+k)\big[ \psi(1-p+k) - \psi(1-\cB+k)\big]}{k!\Gamma(1-p+k)} =\sum_{k=0}^{\infty} \frac{\Gamma(1-\cB+k)\psi(1-p+k)}{k!\Gamma(1-p+k)}\\
& -\sum_{k=0}^\infty \frac{\Gamma(1-\cB+k)\psi(1-\cB+k)}{k!\Gamma(1-p+k)}=\sum_{k=0}^{p-1} \frac{(-1)^{p-k}(p-k-1)!\Gamma(1-\cB+k)}{k!}\\
&+\sum_{k=0}^\infty \frac{\psi(1+k)\Gamma(1-\cB+k+p)}{k!\Gamma(1+p+k)}-\sum_{k=0}^\infty\frac{\Gamma(1-\cB+p+k)\psi(1-\cB+p+k)}{k!\Gamma(1+p+k)}.
\end{align*}

Finally, we can sum ${}_2\phi_{1}(1)$ by the Gauss summation theorem when the denominator parameters are positive. If the denominator parameter is zero or negative, it is easy to compute by shifting the summation index so that 
\begin{align*}
{}_2\phi_{1}\!\left(\!\!\begin{array}{c}
 a,b\\ 1-p
\end{array}\!\!\right)\!=\!\frac{\Gamma(a+p)\Gamma(b+p)}{p!} {}_2F_{1}\!\left(\!\!\begin{array}{c}
 a+p,b+p\\ p+1
\end{array}\!\!\right)\!=\!\frac{\Gamma(a+p)\Gamma(b+p)\Gamma(1-a-b-p)}{\Gamma(1-a)\Gamma(1-b)}.
\end{align*}
Substituting these summations into the first relation of the proof, we obtain the desired result.
\end{proof}

The following corollary deals again with the $r=2$ case of \eqref{eq:duality-no-delta}. In this setting, we can convert a special combination of digamma series into a single ${}_2F_1$ series.

\begin{corollary}\label{cr:euler-digamma-identity}
Let $\cB = (b_1,b_2)$. For $|z|<1$, we have the following identity:
\begin{multline}\label{eq:euler-digamma-identity}
 \frac{1}{(1-z)^{b_1+b_2-1}}\sum_{k=0}^\infty \frac{(1-b_1)_k(1-b_2)_k}{(k!)^2}(\psi(1-b_1+k)+\psi(1-b_2+k)-2\psi(1+k))z^k \\
+\sum_{k=0}^\infty \frac{(b_1)_k(b_2)_k}{(k!)^2} (2\psi(1+k)-\psi(b_1+k)-\psi(b_2+k))z^k 
\\=\pi(\cot \pi(b_1)+\cot \pi(b_2))\:{}_2F_{1}\!\left(\!\!\begin{array}{c}
 b_1,b_2\\1
\end{array}\,\middle| z\right).
\end{multline}
\end{corollary}
\begin{proof}
See the Appendix.
\end{proof}

The proof of the preceding corollary is placed in the Appendix
because we next prove a generalization by employing a different method.

\begin{theorem}
Let $b_{1},b_{2},c\in\mathbb{C}$ be such that all expressions below are non-singular. Then  for  $|z|<1$ the following holds:
\begin{multline}\label{eq:euler-digamma-generalization}
(1-z)^{\,c-b_{1}-b_{2}}\sum_{k=0}^{\infty}
\frac{(c-b_{1})_k(c-b_{2})_k}{(c)_k\,k!}
\bigl(\psi(c-b_{1}+k)+\psi(c-b_{2}+k)-2\psi(c+k)\bigr)z^k
\\
+\sum_{k=0}^{\infty}
\frac{(b_{1})_k(b_{2})_k}{(c)_k\,k!}
\bigl(2\psi(c+k)-\psi(b_{1}+k)-\psi(b_{2}+k)\bigr)z^k
\\
=
\bigl(\psi(c-b_{1})+\psi(c-b_{2})-\psi(b_{1})-\psi(b_{2})\bigr)
{}_2F_1\!\left(\begin{matrix} b_{1},b_{2}\\ c \end{matrix}\middle| z\right).
\end{multline}
\end{theorem}
\begin{proof}
In this proof, write \(F(a,b;c;z)={}_2F_1(a,b;c;z)\).
We start from Euler's transformation
\begin{equation}\label{eq:euler-transformation}
F\big(b_1,b_2;c;z\big)
=
(1-z)^{c-b_{1}-b_{2}}
F\big(c-b_1,c-b_2;c;z\big).
\end{equation}
A simple computation shows that 
$$
(\partial_{b_{1}}+\partial_c)(1-z)^{c-b_{1}-b_{2}}=(\partial_{b_{2}}+\partial_c)(1-z)^{c-b_{1}-b_{2}}=0,
$$
so that the differential operators $\partial_{b_{1}}+\partial_c$ and $\partial_{b_{2}}+\partial_c$ annihilate the function $(1-z)^{c-b_{1}-b_{2}}$. Therefore, applying $\partial_{b_{1}}+\partial_c$ to \eqref{eq:euler-transformation}, we obtain
\begin{equation}\label{eq:euler-parameter-derivative}
(\partial_{b_{1}}+\partial_c)F(b_{1},b_{2};c;z)
=
(1-z)^{c-b_{1}-b_{2}}
(\partial_{b_{1}}+\partial_c)F(c-b_{1},c-b_{2};c;z).    
\end{equation}
Now, 
differentiate the hypergeometric series $F(b_{1},b_{2};c;z)$  term by term using the following relations
\[
\partial_x (x)_k = (x)_k\bigl(\psi(x+k)-\psi(x)\bigr),~~~~\partial_x [1/(x)_k] = [1/(x)_k]\bigl(\psi(x)-\psi(x+k)\bigr). 
\]
We get
\[
(\partial_{b_{1}}+\partial_c)F(b_{1},b_{2};c;z)
=
\sum_{k=0}^{\infty}\frac{(b_{1})_k(b_{2})_k}{(c)_k\,k!}
\bigl(\psi(b_{1}+k)-\psi(b_{1})-\psi(c+k)+\psi(c)\bigr)z^k,
\]
and 
\begin{multline*}
(\partial_{b_{1}}+\partial_{c})F(\underbrace{c-b_{1}}_{t},\underbrace{c-b_{2}}_{s};c;z)
=\partial_{t}F(t,s;c;z)(-1)+\partial_{t}F(t,s;c;z)+\partial_{s}F(t,s;c;z) \\+\partial_{c}F(t,s;c;z) =\partial_{s}F(t,s;c;z)+\partial_{c}F(t,s;c;z) \\
=
\sum_{k=0}^{\infty}\frac{(c-b_{1})_k(c-b_{2})_k}{(c)_k\,k!}
\bigl(\psi(c-b_{2}+k)-\psi(c-b_{2})-\psi(c+k)+\psi(c)\bigr)z^k.
\end{multline*}
Hence, substituting into \eqref{eq:euler-parameter-derivative}, we get
\begin{multline}\label{eq:euler-derivative-b1}
(1-z)^{c-b_{1}-b_{2}}
\sum_{k=0}^{\infty}\frac{(c-b_{1})_k(c-b_{2})_k}{(c)_k\,k!}
\bigl(\psi(c-b_{2}+k)-\psi(c+k)\bigr)z^k
\\
+\sum_{k=0}^{\infty}\frac{(b_{1})_k(b_{2})_k}{(c)_k\,k!}
\bigl(\psi(c+k)-\psi(b_{1}+k)\bigr)z^k
=
\bigl(\psi(c-b_{2})-\psi(b_{1})\bigr)F(b_{1},b_{2};c;z).
\end{multline}

By symmetry, applying \(\partial_{b_{2}}+\partial_c\) to \eqref{eq:euler-transformation} gives
\begin{multline}\label{eq:euler-derivative-b2}
(1-z)^{c-b_{1}-b_{2}}
\sum_{k=0}^{\infty}\frac{(c-b_{1})_k(c-b_{2})_k}{(c)_k\,k!}
\bigl(\psi(c-b_{1}+k)-\psi(c+k)\bigr)z^k
\\
+\sum_{k=0}^{\infty}\frac{(b_{1})_k(b_{2})_k}{(c)_k\,k!}
\bigl(\psi(c+k)-\psi(b_{2}+k)\bigr)z^k
=
\bigl(\psi(c-b_{1})-\psi(b_{2})\bigr)F(b_{1},b_{2};c;z).
\end{multline}
Adding \eqref{eq:euler-derivative-b1} and \eqref{eq:euler-derivative-b2}, we obtain the claimed formula
\eqref{eq:euler-digamma-generalization}.
\end{proof}

\begin{remark}
For \(c=1\) in \eqref{eq:euler-digamma-generalization}, one has \((c)_k=(1)_k=k!\), and the right-hand side becomes
\[
\psi(1-b_1)-\psi(b_1)+\psi(1-b_2)-\psi(b_2).
\]
Using the reflection formula
\[
\psi(1-x)-\psi(x)=\pi\cot(\pi x),
\]
we recover \eqref{eq:euler-digamma-identity}.
\end{remark}
By equating the coefficients  of  $z^k$ in  \eqref{eq:euler-digamma-generalization}, we arrive at the following summation theorem.

\begin{theorem}
Under the assumptions of the preceding theorem, for each integer $n\ge0$, the following summation formula holds:
\begin{multline}\label{eq:euler-coefficient-sum}
\sum_{k=0}^{n} \frac{(b_{1}+b_{2}-c)_{n-k}(c-b_{1})_k(c-b_{2})_k}{(c)_k (n-k)!k!}(\psi(c-b_{1}+k)+\psi(c-b_{2}+k)-2\psi(c+k))\\
= \frac{(b_{1})_n(b_{2})_n}{(c)_n n!} [\psi(b_{1}+n)+\psi(b_{2}+n)-2\psi(c+n)+\psi(c-b_{1})+\psi(c-b_{2})-\psi(b_{1})-\psi(b_{2}) ].
\end{multline}
\end{theorem}
\begin{proof}
Using the binomial expansion 
\begin{align*}
\frac{1}{(1-z)^{b_{1}+b_{2}-c}} = \sum_{n=0}^{\infty} \frac{(b_{1}+b_{2}-c)_n}{n!} z^n
\end{align*}
and the Cauchy product in \eqref{eq:euler-digamma-generalization}, we get \eqref{eq:euler-coefficient-sum} by equating coefficients at equal powers of $z$.
\end{proof}

\begin{remark}
The summation formula \eqref{eq:euler-coefficient-sum} simplifies when $c=1$ in view of the reflection formula $\psi(1-z)-\psi(z)=\pi \cot (\pi z)$ to the following form:
\begin{multline*}
\sum_{k=0}^{n} \frac{(b_{1}+b_{2}-1)_{n-k}(1-b_{1})_k(1-b_{2})_k}{(n-k)! (k!)^2}(\psi(1-b_{1}+k)+\psi(1-b_{2}+k)-2\psi(1+k))\\
= \frac{(b_{1})_n(b_{2})_n}{(n!)^2} [\psi(1-b_{1}-n)+\psi(1-b_{2}-n)-2\psi(1+n)].
\end{multline*}
\end{remark}

In the following two examples, we take particular values of $b_{1}$ and $b_{2}$. 

\begin{example}
Setting $b_{1}=b_{2}=1/4$ in formula \eqref{eq:euler-digamma-identity} and using   \cite[(74), p.576]{Brychkov_2008}, we obtain a combination of digamma series expressible by the complete elliptic integral of the first kind (denoted by $K$), namely 
\begin{multline*}
\sqrt{1-z} \sum_{k=0}^{\infty} \frac{\left(\frac{3}{4}\right)_k^2}{(k!)^2} 
\left(\!\psi\left(\frac{3}{4}+k\right)-\psi(1+k)\!\right) z^k 
+ \sum_{k=0}^{\infty} \frac{\left(\frac{1}{4}\right)_k^2}{(k!)^2} 
\left(\!\psi(1+k)-\psi\left(\frac{1}{4}+k\right)\!\right) z^k 
\\=\pi {}_2F_1 \left(\frac{1}{4},\frac{1}{4};1;z \right) =2K \left( \sqrt{\frac{1}{2}-\frac{1}{2} (1-z)^{1/2}}\right).
\end{multline*}
Further, using  \cite[(112, p.580)]{Brychkov_2008}, we get  the following summation formula
\begin{multline*}
\frac{3}{2\sqrt{2}} \sum_{k=0}^{\infty} \frac{\left(\frac{3}{4}\right)_k^2}{(k!)^2} 
\left(\!\psi\left(\frac{3}{4}+k\right)-\psi(1+k)\!\right) 
+ \frac{\left(\frac{1}{4}\right)_k^2}{(k!)^2} 
\left(\!\psi(1+k)-\psi\left(\frac{1}{4}+k\right)\!\right)  \left(-\frac{1}{8}\right)^k 
\\=\pi {}_2F_1 \left(\frac{1}{4},\frac{1}{4};1;-\frac{1}{8} \right) =2^{-5/4}\pi^{-1/2} \Gamma^2\left(\frac{1}{4}\right).
\end{multline*}
\end{example}

\begin{example}
With $b_{1}=b_{2}=1/3$ in \eqref{eq:euler-coefficient-sum}, we have
\begin{align*}
\sum_{k=0}^{n} \frac{(-1/3)_{n-k}(2/3)_k^2}{(n-k)!(k!)^2} \left(\psi(2/3+k)-\psi(1+k)\right)
=\frac{(1/3)_n^2}{(n!)^2}\left[\psi(1/3+n)-\psi(1+n)+\frac{\pi}{\sqrt{3}}\right].
\end{align*}
\end{example}

\section{Finite sums involving digamma functions}\label{sec:finite-digamma-sums}
In this section, we retain the notation from the previous section as given in \eqref{eq:integer-vector-notation}.  
One consequence of the proof of the duality relation \eqref{eq:duality-relation} is the following identity stated in \cite[Theorem~6.2]{Cetinkaya_Karp_2021}:
\begin{align}\label{eq:finite-duality-relation}
\sum_{i=1}^r \frac{(1-\frakb+\alpha_i)_{\bfm+k}}{(\alpha_i-\fraka_{[i]})_{\bfn_{[i]}+k+1} (k+n_i)!} {}_{2r}F_{2r-1} \left(\begin{array}{c}
 -k-n_i,\frakb-\alpha_i, \fraka_{[i]}-\alpha_i-\bfn_{[i]}-k\\ \frakb-\alpha_i-\bfm-k, 1+\fraka_{[i]}-\alpha_i
\end{array} \middle| 1\right)=q_e(k),
\end{align}
where all terms with $k+n_i<0$ vanish by convention, $q_{e}\equiv0$ for $e=-1$ and is a polynomial of degree $e$ for $e\ge0$. This polynomial is defined by the recurrence
\begin{align}\label{eq:q-e-recurrence}
q_0(y) = 1, \quad q_e(y)=\frac{1}{e} \sum_{j=1}^{e} \frac{(-1)^{j+1}}{j+1} Q_j(y) q_{e-j}(y),~~e=1,2,\ldots,
\end{align}
where, keeping the notation $\fraka=(\alpha_1,\ldots,\alpha_r)$ and $\frakb=(\beta_1,\ldots,\beta_r)$, the polynomial $Q_j$ is given by
\begin{align*}
Q_j(y)\! =\! \sum_{i=1}^{r}\bigl[\mathbf{B}_{j+1}(-\alpha_i-y)-\mathbf{B}_{j+1}(1-\beta_i-y)+\mathbf{B}_{j+1}(1-\beta_i+m_i)-\mathbf{B}_{j+1}(1-\alpha_i+n_i)\bigr]
\end{align*}
and $\mathbf{B}_{n}(x)$ stands for the standard Bernoulli polynomials \cite[Formula (1)]{Norlund_1961}. 

If one (or more) of the differences $\alpha_{i}-\alpha_{j}$ is an integer, then identity \eqref{eq:finite-duality-relation} fails, but the right-hand side remains well-defined. Below, we explore the degenerate form of the left-hand side when one of such differences is an integer. This leads to an evaluation of certain digamma sums in terms of expressions involving Bernoulli polynomials and finite sums of hypergeometric functions. Before stating our main theorem, it is convenient to introduce the following notation:
\begin{equation}\label{eq:digamma-increment}
\psi(x;j)=\psi(x+j)-\psi(x)=\sum_{k=0}^{j-1}\frac{1}{x+k},~~\psi(\mathbf{x};j)=\sum_{\ell=1}^{d}\psi(x_{\ell};j),   
\end{equation}
where $\mathbf{x}=(x_1,\ldots,x_d)$. At \(x=-n\), this difference has a finite continuous extension when \(0\le j\le n\). Indeed, suppose $n\ge{j}\ge0$ are integers. Then $\psi(-n;j)=\psi(-n+j)-\psi(-n)$ can be defined by continuity in view of the asymptotic formula
$$
\psi(-n+\epsilon)=-\frac{1}{\epsilon}+H_{n}-\gamma+O(\epsilon),~~~\epsilon\to0,
$$
where $H_{n}$ is $n$-th harmonic number.  This formula implies that defining 
\begin{equation}\label{eq:digamma-increment-negative}
\psi(-n;j)=H_{n-j}-H_{n}    
\end{equation}
makes $\psi(x;j)$ continuous at $x=-n$. Expressions of this type appearing in $\Theta_{j}$, $\hat{\Theta}_{j}$ below are to be understood in this way. If $j>n\ge0$, then $\psi(x;j)$ has a simple pole at $x=-n$ with asymptotics
\begin{equation*}
\psi(-n+\epsilon;j)=\psi(j-n)-\psi(-n+\epsilon)=\frac{1}{\epsilon}+H_{j-n-1}-H_{n}+O(\epsilon),~~\epsilon\to0.    
\end{equation*}
This formula implies that for $j>n$ and $i>m$ the product $\psi(-n;j)(-m)_{i}$  can be defined by the limit
\begin{equation}\label{eq:digamma-pochhammer-limit}
\psi(-n+\epsilon;j)(-m+\epsilon)_{i}\to (-1)^{m}m!(i-m-1)!~\text{as}~\epsilon\to0.
\end{equation}

We further remind the reader that each expression of the form $\psi(-n)/\Gamma(-n)$, $n\in\mathbb{N}_{0}$, is understood as $(-1)^{n+1}n!$ according to \eqref{eq:psi-over-gamma-negative}.
Then we have the following
\begin{theorem}\label{th:finite-digamma-sum}
Fix $p \in \mathbb{Z}$ and an integer $r\ge2$. Suppose  $\fraka\in \mathbb{C}^{r}$ satisfies
$\alpha_2=\alpha_1+p$ and $\alpha_i-\alpha_j\notin\mathbb{Z}$ for every pair \(\{i,j\}\ne\{1,2\}\). 
Let $\frakb \in \mathbb{C}^{r}$  and define the parameter vectors 
$$
\cA=1-\fraka_{[1,2]}+\alpha_1+p,~~\cB=1-\frakb+\alpha_1+p.
$$ 
Then for each integer $k \geq -m_{\min}$, the following identity holds:
\begin{align}\label{eq:finite-digamma-sum}
&\cG_{k}\sum_{j\ge0}\cH_j\Theta_j 
-\hat{\cG}_{k} \sum_{j\ge0}\hat{\cH}_j \hat{\Theta}_j =q_e(k) \notag\\[7pt]
&-\sum_{i=3}^r \frac{(1-\frakb+\alpha_i)_{\bfm+k}}{(\alpha_i-\fraka_{[i]})_{\bfn_{[i]}+k+1} (k+n_i)!}\: {}_{2r}F_{2r-1} \!\left(\!\!\begin{array}{c}
 -k-n_i,\frakb-\alpha_i, \fraka_{[i]}-\alpha_i-\bfn_{[i]}-k\\ \frakb-\alpha_i-\bfm-k, 1+\fraka_{[i]}-\alpha_i
\end{array}\!\right)\notag\\[7pt]
&+\cG_{k}\cdot\Psi_k\cdot{}_{2r}F_{2r-1}\!\left(\!\!\begin{array}{c}
 -k-n_2,-p-n_1-k,1-\cB,1-\cA-\bfn_{[1,2]}-k\\ 1-p,1-\cB-\bfm-k,2-\cA
\end{array} \!\right)\notag\\[7pt]
&+\hat{\cG}_{k}\cdot\hat{\Psi}_k\cdot{}_{2r}F_{2r-1}\!\left(\!\!\begin{array}{c}
 -k-n_1,p-n_2-k,1-\cB+p,1-\cA+p-\bfn_{[1,2]}-k\\ 1+p,1-\cB+p-\bfm-k, 2-\cA+p
\end{array} \!\right),
\end{align}
where $q_e(k)$ is defined in \eqref{eq:q-e-recurrence} and 
\begin{align}
&\cG_{k}=\frac{(-1)^{p}\Gamma(\cB+\bfm+k)\Gamma(\cA-1)}{\Gamma(\cB)\Gamma(\cA+\bfn_{[1,2]}+k)\Gamma(1-p)\Gamma(p+{n}_1+k+1)(k+n_2)!}, \notag\\[7pt] 
&\hat{\cG}_{k}=\frac{(-1)^{p}\Gamma(\cB-p+\bfm+k)\Gamma(\cA-p-1)}{\Gamma(\cB-p)\Gamma(\cA-p+\bfn_{[1,2]}+k)\Gamma(1+p)\Gamma(-p+{n}_2+k+1)(k+n_1)!}, \notag\\[7pt]
&\cH_j=\frac{(-k-n_2)_j (-p-n_1-k)_j (1-\cB)_j (1-\cA-\bfn_{[1,2]}-k)_j}{(1-p)_{j}(1-\cB-\bfm-k)_{j}(2-\cA)_{j}  j!},\notag \\[7pt]
&\hat{\cH}_j = \frac{(-k-n_1)_j (p-n_2-k)_j (1-\cB+p)_j (1-\cA+p-\bfn_{[1,2]}-k)_j}{(1+p)_{j}(1-\cB+p-\bfm-k)_{j}(2-\cA+p)_{j}j!}, \notag \\[7pt] 
&\Theta_j=-\psi(1-\cB;j)-\psi(-p-{n}_1-k;j)
-\psi(1-\cA-\bfn_{[1,2]}-k;j)\notag\\
&\quad+\psi(1-\cB-\bfm-k;j)
+\psi(1-p;j)+\psi(2-\cA;j),  \label{eq:theta-j-definition}\\[7pt]
&\hat{\Theta}_j= \psi(p-{n}_2-k;j)-\psi(1+p;j),  \label{eq:hat-theta-j-definition}\\[7pt]
&\Psi_k= \psi(\cB)-\psi(\cB+\bfm+k)+\psi(\cA+\bfn_{[1,2]}+k)-\psi(\cA-1)
\label{eq:psi-k-definition}\\
&\quad+\psi(p+{n}_1+k+1)-\psi(1-p), \notag \\[7pt]
&\hat{\Psi}_k= \psi(-p+{n}_2+k+1)-\psi(1+p). \label{eq:hat-psi-k-definition}
\end{align}
Note that all sums in \eqref{eq:finite-digamma-sum} terminate in view of definitions of $\cH_{j}$, $\hat{\cH}_{j}$ and  the singular products of the type $0\cdot\infty$ are to be  interpreted as common limit as $\alpha_2-\alpha_1\to{p}$, which corresponds to the regularization described above. The terms in the sum $\sum_{i=3}^{r}$ on the right-hand side of \eqref{eq:finite-digamma-sum} with $k+n_i<0$ vanish by convention.
\end{theorem}
\begin{proof}

Rewrite  \eqref{eq:finite-duality-relation} as follows: 
\begin{align}\label{eq:finite-duality-split}
&\frac{(1-\frakb+\alpha_1)_{\bfm+k}}{(\alpha_1-\fraka_{[1]})_{\bfn_{[1]}+k+1} (k+n_1)!} {}_{2r}F_{2r-1} \left(\begin{array}{c}
 -k-n_1,\frakb-\alpha_1, \fraka_{[1]}-\alpha_1-\bfn_{[1]}-k\\ \frakb-\alpha_1-\bfm-k, 1+\fraka_{[1]}-\alpha_1
\end{array} \right) + \notag\\
&\frac{(1-\frakb+\alpha_2)_{\bfm+k}}{(\alpha_2-\fraka_{[2]})_{\bfn_{[2]}+k+1} (k+n_2)!} {}_{2r}F_{2r-1} \left(\begin{array}{c}
 -k-n_2,\frakb-\alpha_2, \fraka_{[2]}-\alpha_2-\bfn_{[2]}-k\\ \frakb-\alpha_2-\bfm-k, 1+\fraka_{[2]}-\alpha_2
\end{array} \right)=q_e(k) \notag\\
&-\sum_{i=3}^r \frac{(1-\frakb+\alpha_i)_{\bfm+k}}{(\alpha_i-\fraka_{[i]})_{\bfn_{[i]}+k+1} (k+n_i)!} {}_{2r}F_{2r-1} \left(\begin{array}{c}
 -k-n_i,\frakb-\alpha_i, \fraka_{[i]}-\alpha_i-\bfn_{[i]}-k\\ \frakb-\alpha_i-\bfm-k, 1+\fraka_{[i]}-\alpha_i
\end{array} \right).
\end{align}
It is valid for $\fraka, \frakb \in \mathbb{C}^{r}$ whenever $\alpha_i - \alpha_j \notin \mathbb{Z}$. As above, it is enough to treat $p\ge0$, since the case $p<0$ follows by interchanging the indices $1$ and $2$. Put  $\alpha_2 =\alpha_1+p+\epsilon$. The left-hand side of \eqref{eq:finite-duality-split} then takes the form $f_1(\epsilon)+f_2(\epsilon)$, where
\begin{align*}
&f_1(\epsilon)=g_1(\epsilon){}_{2r}F_{2r-1}\!\left(\!\!\begin{array}{c}
 -k-n_1,\frakb-\alpha_1,p+\epsilon-n_2-k,\fraka_{[1,2]}-\alpha_1-\bfn_{[1,2]}-k\\ \frakb-\alpha_1-\bfm-k,1+p+\epsilon,~ 1+\fraka_{[1,2]}-\alpha_1
\end{array}\!\right), \\
&f_2(\epsilon)=g_2(\epsilon){}_{2r}F_{2r-1}\!\left(\!\!\begin{array}{c}
 -k-n_2,1-\cB-\epsilon,-p-\epsilon-n_1-k,1-\cA-\epsilon-\bfn_{[1,2]}-k\\ 1-\cB-\epsilon-\bfm-k,1-p-\epsilon,~ 2-\cA-\epsilon
\end{array}\!\right),
\end{align*} 
with
\begin{align*}
&g_1(\epsilon) = \frac{(1-\frakb+\alpha_1)_{\bfm+k}}{(-p-\epsilon)_{{n}_2+k+1} (\alpha_1-\fraka_{[1,2]})_{\bfn_{[1,2]}+k+1}(k+n_1)!}\\
&=\!\frac{\Gamma(1-\frakb+\alpha_1+\bfm+k)\Gamma(-p-\epsilon)\Gamma(\alpha_1-\fraka_{[1,2]})}{\Gamma(1-\frakb+\alpha_1)\Gamma(1-p-\epsilon+{n}_2+k)\Gamma(\alpha_1-\fraka_{[1,2]}+\bfn_{[1,2]}+k+1)\Gamma(k+n_1+1)} ,\\[7pt]
& g_2(\epsilon) = \frac{(1-\frakb+\alpha_2)_{\bfm+k}}{(p+\epsilon)_{{n}_1+k+1} (\alpha_2-\fraka_{[1,2]})_{\bfn_{[1,2]}+k+1}(k+n_2)!}  \\
&=\frac{\Gamma(\cB+\epsilon+\bfm+k)\Gamma(p+\epsilon)\Gamma(\cA-1+\epsilon)}{\Gamma(\cB+\epsilon)\Gamma(p+\epsilon+{n}_1+k+1)\Gamma(\cA+\epsilon+\bfn_{[1,2]}+k)\Gamma(k+n_2+1)}.
\end{align*}
Using the reflection formula $\Gamma(z)\Gamma(1-z)=\pi/\sin(\pi z)$ and the relation $\sin[\pi(z+p)] = (-1)^p \sin(\pi z)$ in the view of \eqref{eq:regularized-hypergeometric} we will have 
\begin{align*}
f_1(\epsilon)= \frac{(-1)^{p+1} \pi}{\sin(\pi\epsilon)} h_1(\epsilon), \quad
f_2(\epsilon)= \frac{(-1)^{p} \pi}{ \sin(\pi\epsilon)}h_2(\epsilon),
\end{align*}
where
\begin{align*}
&h_1(\epsilon)=\frac{\Gamma(1-\frakb+\alpha_1+\bfm+k)\Gamma(\alpha_1-\fraka_{[1,2]})}{\Gamma(1-\frakb+\alpha_1)\Gamma(-p-\epsilon+{n}_2+k+1)\Gamma(\alpha_1-\fraka_{[1,2]}+\bfn_{[1,2]}+k+1)(k+n_1)!}\\
&\qquad\times\frac{1}{\Gamma(1+p+\epsilon)}{}_{2r}F_{2r-1}\!\left(\!\!\begin{array}{c}
 -k-n_1,\frakb-\alpha_1,p+\epsilon-n_2-k,\fraka_{[1,2]}-\alpha_1-\bfn_{[1,2]}-k\\ \frakb-\alpha_1-\bfm-k,1+p+\epsilon,~ 1+\fraka_{[1,2]}-\alpha_1
\end{array}\!\right), \\
&h_2(\epsilon)=\frac{\Gamma(\cB+\epsilon+\bfm+k)\Gamma(\cA-1+\epsilon)}{\Gamma(\cB+\epsilon)\Gamma(p+\epsilon+{n}_1+k+1)\Gamma(\cA-1+\epsilon+\bfn_{[1,2]}+k+1)\Gamma(1-p-\epsilon)(k+n_2)!}\\
&\qquad\times{}_{2r}F_{2r-1}\!\left(\!\!\begin{array}{c}
 -k-n_2,1-\cB-\epsilon,-p-\epsilon-n_1-k,1-\cA-\epsilon-\bfn_{[1,2]}-k\\ 1-\cB-\epsilon-\bfm-k,1-p-\epsilon,~ 2-\cA-\epsilon
\end{array}\!\right).
\end{align*} 

Employing Taylor expansion for $h_1(\epsilon)$ and $h_2(\epsilon)$ as in \eqref{eq:duality-limit}, we next compute $h'_1(0)$ using the relation \eqref{eq:reciprocal-gamma-derivative}. This yields
\begin{align*}
&h'_1(0) = \frac{\Gamma(1-\frakb+\alpha_1+\bfm+k)\Gamma(\alpha_1-\fraka_{[1,2]})}{\Gamma(1+p)\Gamma(1-\frakb+\alpha_1)\Gamma(1-p+{n}_2+k)\Gamma(\alpha_1-\fraka_{[1,2]}+\bfn_{[1,2]}+k+1)(k+n_1)!} \\
&\bigg[(\psi(1-p+{n}_2+k)-\psi(1+p)){}_{2r}F_{2r-1}\!\left(\!\!\begin{array}{c}
 -\!k-\!n_1,\frakb\!-\!\alpha_1,p\!-\!n_2\!-\!k,\fraka_{[1,2]}\!-\!\alpha_1\!-\!\bfn_{[1,2]}\!-\!k\\ \frakb-\alpha_1-\bfm-k,1+p,~ 1+\fraka_{[1,2]}-\alpha_1
\end{array}\!\!\right)\\
&+\sum_{j=0}^{k+n_1}\frac{(-k-n_1)_j (\frakb-\alpha_1)_j (p-n_2-k)_j (\fraka_{[1,2]}-\alpha_1-\bfn_{[1,2]}-k)_j \hat{\Theta}_{j}}{(\frakb-\alpha_1-\bfm-k)_j (1+p)_j (1+\fraka_{[1,2]}-\alpha_1)_j j!}\bigg],
\end{align*} 
where $\hat{\Theta}_{j}$ is given by \eqref{eq:hat-theta-j-definition}. In a similar fashion for $h_2'(0)$ we obtain 
\begin{align*}
&h'_2(0)= \frac{\Gamma(\cB+\bfm+k)\Gamma(\cA-1)}{\Gamma(1-p)\Gamma(\cB)\Gamma(p+{n}_1+k+1)\Gamma(\cA+\bfn_{[1,2]}+k)(k+n_2)!}\\
&\times\!\bigg[\big(\psi(\cB+\bfm+k)-\psi(\cB)+\psi(\cA-1)-\psi(\cA+\bfn_{[1,2]}+k)-\psi(p+{n}_1+k+1) \\
& +\psi(1-p)\big){}_{2r}F_{2r-1}\!\left(\!\!\begin{array}{c}
 -k-n_2,\frakb-\alpha_1-p,-p-n_1-k,\fraka_{[1,2]}-\alpha_1-p-\bfn_{[1,2]}-k\\ \frakb-\alpha_1-p-\bfm-k,1-p,~ 1+\fraka_{[1,2]}-\alpha_1-p
\end{array}\!\right)
\\
&+\sum_{j\ge0}\frac{(-k-n_2)_j (\frakb-\alpha_1-p)_j (-p-n_1-k)_j (\fraka_{[1,2]}-\alpha_1-p-\bfn_{[1,2]}-k)_j \Theta_{j}}{(\frakb-\alpha_1-p-\bfm-k)_j (1-p)_j (1+\fraka_{[1,2]}-\alpha_1-p)_j j!}\bigg],
\end{align*} 
where $\Theta_{j}$ is given by \eqref{eq:theta-j-definition}. Finally, we employ a calculation similar to \eqref{eq:duality-limit}, replacing  $h'_1(0)$ and $h'_2(0)$ by the above expressions. Substituting everything back into \eqref{eq:finite-duality-split} and taking the limit as $\epsilon\to0$ yields \eqref{eq:finite-digamma-sum}.
\end{proof}

\begin{remark}
Care is needed when differentiating the factors $(p+\epsilon-n_2-k)_j$ and $(-p-\epsilon-n_1-k)_j$ at $\epsilon=0$. The Pochhammer symbol itself is a polynomial and hence has no singularity. For any $x$,
\[
\left.\frac{d}{d\epsilon}(x+\epsilon)_j\right|_{\epsilon=0}
=\sum_{m=0}^{j-1}\prod_{\substack{i=0\\i\ne m}}^{j-1}(x+i).
\]
If \(x\notin\{0,-1,\ldots,-j+1\}\), this equals
\((x)_j[\psi(x+j)-\psi(x)]\). If \(x=-m\) with \(0\le m\le j-1\), its value is instead
\[
(-1)^m m!(j-m-1)!.
\]
These formulas apply to both Pochhammer factors above and avoid division by a vanishing factor.
\end{remark}

Specializing the parameters gives the following more concise
summation formula.
\begin{corollary}
Suppose  $\fraka\in \mathbb{C}^{r}$ satisfies $\alpha_2=\alpha_1$ and $\alpha_i-\alpha_j\notin\mathbb{Z}$ for every pair $\{i,j\}\ne\{1,2\}$, $\bfn\in\mathbb{Z}^r$ satisfies $n_1=n_2=n$, and $\frakb \in \mathbb{C}^{r}$. Further, denote $\cA=1-\fraka_{[1,2]}+\alpha_1$ and  $\cB=1-\frakb+\alpha_1$.
Then for each integer $k$ such that $k+n\ge0$, the following identity holds 
\begin{align}\label{eq:equal-shift-summation}
&\frac{\Gamma(\cB+\bfm+k)\Gamma(\cA-1)}{\Gamma(\cB)\Gamma(\cA+\bfn_{[1,2]}+k)[(k+n)!]^2}\sum_{j=0}^{k+n}\frac{[(-k-n)_j]^2(1-\cB)_j (1-\cA-\bfn_{[1,2]}-k)_j}{(1-\cB-\bfm-k)_{j}(2-\cA)_{j}(j!)^2} E_j
\notag\\
&=q_{e}(k)-\sum_{i=3}^r \frac{(1-\frakb+\alpha_i)_{\bfm+k}}{(\alpha_i-\fraka_{[i]})_{\bfn_{[i]}+k+1} (k+n_i)!}\: {}_{2r}F_{2r-1} \!\left(\!\!\begin{array}{c}
 -k-n_i,\frakb-\alpha_i, \fraka_{[i]}-\alpha_i-\bfn_{[i]}-k\\ \frakb-\alpha_i-\bfm-k, 1+\fraka_{[i]}-\alpha_i
\end{array}\!\right),
\end{align}
where
\begin{align}\label{eq:e-j-definition}
E_j=
& \psi(1\!-\!\cB\!-\!\bfm\!-\!k\!+\!j)-\psi(1
 \!-\!\cB\!+\!j)+\psi(2\!-\!\cA\!+\!j)
 -\psi(1\!-\!\cA\!-\!\bfn_{[1,2]}\!-\!k\!+\!j)\notag\\
 &+2H_j-2H_{k+n-j}.
\end{align}
\end{corollary}

\begin{proof}
Taking $p=0$, $n_1=n_2$ in Theorem~\ref{th:finite-digamma-sum},
we see that
\begin{align*}
&G_k:=\cG_k=\hat{\cG}_{k}=\frac{\Gamma(\cB+\bfm+k)\Gamma(\cA-1)}{\Gamma(\cB)\Gamma(\cA+\bfn_{[1,2]}+k)[(k+n)!]^2}, \notag\\[7pt]
&H_j:=\cH_j=\hat{\cH}_{j}=\frac{[(-k-n)_j]^2(1-\cB)_j (1-\cA-\bfn_{[1,2]}-k)_j}{(1-\cB-\bfm-k)_{j}(2-\cA)_{j}(j!)^2}.
\end{align*}
Hence, the left-hand side of \eqref{eq:finite-digamma-sum} takes the form $G_k\sum_{j\ge0}H_j(\Theta_j-\hat{\Theta}_{j})$. Separating the part independent of $j$ we compute the difference $\Theta_j-\hat{\Theta}_{j}$ by an application of \eqref{eq:digamma-increment} and \eqref{eq:digamma-increment-negative} as follows:
\begin{equation*} 
\Theta_j-\hat\Theta_j=S_k+E_j,
\end{equation*}
where $E_{j}$ is defined in \eqref{eq:e-j-definition}, and the $j$-independent part is
\begin{equation*}
S_k={}
 \psi(\cB)-\psi(\cB+\bfm+k)
 +\psi(\cA+\bfn_{[1,2]}+k)-\psi(\cA-1)
 +2H_{k+n} =\Psi_{k}+\hat{\Psi}_k.
 \end{equation*}
The ultimate equality follows from the definitions \eqref{eq:psi-k-definition},\eqref{eq:hat-psi-k-definition}.
In view of these relations, the left-hand side of formula \eqref{eq:finite-digamma-sum} takes the form 
$$
G_k\sum_{j=0}^{k+n}H_j(\Theta_j-\hat\Theta_j)
=G_k\sum_{j=0}^{k+n}H_j(S_k+E_j)  =G_k(\Psi_{k}+\hat{\Psi}_k)\Phi_k+G_k\sum_{j=0}^{k+n}H_jE_j,
$$
where 
$$
 \Phi_k:=
 {}_{2r}F_{2r-1}\!\left(
 \begin{matrix}
 -k-n,-k-n,1-\cB,
 1-\cA-\bfn_{[1,2]}-k\\
 1,1-\cB-\bfm-k,2-\cA
 \end{matrix}
 \right)
 =\sum_{j=0}^{k+n}H_j.
$$
This shows that the expression $G_k(\Psi_{k}+\hat{\Psi}_k)\Phi_k$ also appearing on the right-hand side of \eqref{eq:finite-digamma-sum} cancels and we are left with \eqref{eq:equal-shift-summation}.
\end{proof}

\begin{corollary}
Under the assumptions of the preceding corollary, let $r=3$, so that $\alpha:=\cA = 1-\alpha_3+\alpha_1$ is a scalar. Then the following identity holds:
\begin{align*}
&\frac{\Gamma(\cB+\bfm+k)\Gamma(\alpha-1)}{\Gamma(\cB)\Gamma(\alpha+n_3+k)[(k+n)!]^2}\sum_{j=0}^{k+n}\frac{[(-k-n)_j]^2(1-\cB)_j (1-\alpha-n_3-k)_j}{(1-\cB-\bfm-k)_{j}(2-\alpha)_{j}(j!)^2} E_j
\notag\\
&=q_{e}(k)- \frac{(\cB-\alpha+1)_{\bfm+k}}{[(1-\alpha)_{n+k+1} ]^2(k+n_3)!}{}_{6}F_{5} \!\left(\!\!\begin{array}{c}
 -k-n_3,\alpha-\cB, \alpha-n-k-1, \alpha-n-k-1\\ \alpha-\cB-\bfm-k, \alpha, \alpha
\end{array}\!\right)
\end{align*}
where
\begin{align*}
E_j=
& \psi(1\!-\!\cB\!-\!\bfm\!-\!k\!+\!j)-\psi(1
 \!-\!\cB\!+\!j)+\psi(2\!-\!\alpha\!+\!j)
 -\psi(1\!-\!\alpha\!-\!n_3\!-\!k\!+\!j)\notag\\
 &+2H_j-2H_{k+n-j}.
\end{align*}
Further simplifications occur if \emph{(1)} $M-N = 1$, then $e=-1$ and $q_{e}(k)=0$; \emph{(2)} if $M-N = 2$, then $e=0$ and $q_{e}(k)=1$;  \emph{(3)} if $M-N = 3$, then $e=1$ and $q_{e}(k)=Q_1(k)/2$.
\end{corollary}

\begin{corollary}
Let $n_1,n_2,m_1,m_2$ be nonnegative integers such that $n_1 \le n_2$, and let $a_1, a_2$ be generic complex parameters. The following identity holds:
\begin{multline}\label{eq:r2-p0-k0-summation}
\sum_{j=0}^{n_1}\frac{(-n_1)_j (-n_2)_j (a_1)_{j}(a_2)_{j}   }{(a_1-m_1)_{j}(a_2-m_2)_{j}  (j!)^2} \Big[\psi(a_1-m_1+j)+\psi(a_2-m_2+j)\\-\psi(a_1+j)-\psi(a_2+j)+2H_j-H_{n_1-j} -H_{n_2-j}\Big] \\-\sum_{j=n_1+1}^{n_2} \frac{(-n_2)_j (a_1)_j (a_2)_j}{(a_1-m_1)_j (a_2-m_2)_j (j!)^2} (-1)^{n_1} n_1! (j-n_1-1)! = \frac{n_1!n_2!}{(1-a_1)_{m_1}(1-a_2)_{m_2}} q_e(0)
\end{multline}
\end{corollary}
\begin{proof}
We take $r=2$, $p=k=0$ in Theorem~\ref{th:finite-digamma-sum}.  As  $r=2$, the sum over $i$ from $3$ to $r$ on the RHS of \eqref{eq:finite-digamma-sum} vanishes and the vector $\cA$ is empty. The product $\cG_0\cH_{j}\Theta_{j}$ then becomes
\begin{equation*}
\frac{(\cB)_{\bfm}}{{n}_{1}!n_{2}!}   
\frac{(-n_2)_j (-n_1)_j (1-\cB)_j }{(1-\cB-\bfm)_{j}(j!)^2}
(\psi(1-\cB-\bfm;j)-\psi(1-\cB;j)-\psi(-{n}_1;j)
+\psi(1;j)).
\end{equation*}
 Hence, by virtue of Theorem \ref{th:finite-digamma-sum} we have
\begin{align*}
\sum_{j\ge0}\frac{(-n_1)_j (-n_2)_j (1-\cB)_j  [\Theta_j-\hat{\Theta}_j-\Psi_0-\hat{\Psi}_0]}{(1-\cB-\bfm)_j  (j!)^2}=\frac{n_1!n_2!}{(\cB)_\bfm} q_e(0),
\end{align*}
where
\begin{align*}
&\Theta_j-\hat{\Theta}_j-\Psi_0-\hat{\Psi}_0= \psi(1-\cB-\bfm;j)-\psi(1-\cB;j)-\psi(-{n}_1;j)
+2\psi(1;j)-\psi(-{n}_2;j)\\
&-\psi(\cB)+\psi(\cB+\bfm)-\psi({n}_1+1)+2\psi(1)-\psi({n}_2+1).
\end{align*}
Substituting the definition \eqref{eq:digamma-increment}, applying \eqref{eq:digamma-increment-negative}  and replacing $1-\cB=(a_1,a_2)$ for brevity, for $0\le{j}\le n_1$ we get 
\begin{align*}
&\Theta_j-\hat{\Theta}_j-\Psi_0-\hat{\Psi}_0\\
&=-\psi(a_1+j)-\psi(a_2+j)+\psi(a_1-m_1+j)+\psi(a_2-m_2+j)+2H_j-H_{n_1-j} -H_{n_2-j}.
\end{align*}
Finally, canceling identical terms in view of the easily verifiable identity 
\begin{multline*}
\psi(a_1)+\psi(a_2)-\psi(1-a_1)-\psi(1-a_2)-\psi(a_1-m_1)-\psi(a_2-m_2)+\psi(1-a_1+m_1)\\+\psi(1-a_2+m_2)=0,
\end{multline*} 
and applying  \eqref{eq:digamma-pochhammer-limit} for $n_1<j\le {n_2}$ we arrive at \eqref{eq:r2-p0-k0-summation}.
\end{proof}

\begin{remark}
Further curious  particular cases of \eqref{eq:r2-p0-k0-summation} can be obtained by observing that
\begin{enumerate}
\item If $m_1+m_2-n_1-n_2=0$, then $e=-1$ and $q_e(0)=0$.
\item If $m_1+m_2-n_1-n_2=1$, then $e=0$ and $q_e(0)=1$.
\item If $m_1+m_2-n_1-n_2=2$, then $e=1$ and $q_e(0)=Q_1(0)/2$.
\end{enumerate}
\end{remark}

\begin{remark}
If we assume $n_1=n_2$ and $m_1+m_2=2n$ in \eqref{eq:r2-p0-k0-summation}, then the second sum vanishes and $q_e(0)=0$, so that
\begin{multline*}
\sum\limits_{j=0}^{n}\frac{[(-n)_{j}]^2(a_1)_{j}(a_2)_{j}}{(a_1-m_1)_{j}(a_2-m_2)_{j}(j!)^2}(\psi(a_1-m_1+j)+\psi(a_2-m_2+j)
\\
-\psi(a_1+j)-\psi(a_2+j)+2(H_{j}-H_{n-j}))=0.
\end{multline*} 
If $m_1+m_2=2n+1$ the right-hand side of the above formula becomes $$
\frac{(n!)^2}{(1-a_{1})_{m_1}(1-a_{2})_{m_2}}.$$
\end{remark}
A different type of corollary can be obtained when both factors $\cG_k$ and $\hat{\cG}_k$ vanish while  the products $\cG_k\Psi_{k}$ and $\hat{\cG}_k\hat{\Psi}_{k}$ remain finite and nonzero. In this case, all digamma series vanish, and we obtain the following purely hypergeometric relation.

\begin{corollary}
Under the assumptions of Theorem~\ref{th:finite-digamma-sum}, let $q=|p|\ge1$, assume that the remaining parameters are generic, and
suppose
\begin{equation*} 
 0\le k+n_1\le q-1,\qquad 0\le k+n_2\le q-1.
 \end{equation*}
Then    
\begin{align}
 &\sum_{i=3}^r
 \frac{(1-\frakb+\alpha_i)_{\bfm+k}}
 {(\alpha_i-\fraka_{[i]})_{\bfn_{[i]}+k+1}(k+n_i)!}\:
 {}_{2r}F_{2r-1}\!\left(
 \begin{matrix}
 -k-n_i,\ \frakb-\alpha_i,\
 \fraka_{[i]}-\alpha_i-\bfn_{[i]}-k\\
 \frakb-\alpha_i-\bfm-k,\
 1+\fraka_{[i]}-\alpha_i
 \end{matrix}\right)
 \notag\\
 &=q_e(k)+C_k\cdot{}{}_{2r}F_{2r-1}\!\left(
 \begin{matrix}
 -k-n_2,-p-k-n_1,1-\cB,
 1-\cA-\bfn_{[1,2]}-k\\
 1-p,1-\cB-\bfm-k,2-\cA
 \end{matrix}\right)
 \notag\\
 &+\widehat C_k\cdot{}_{2r}F_{2r-1}\!\left(
 \begin{matrix}
 -k-n_1,p-k-n_2,1-\cB+p,
 1-\cA+p-\bfn_{[1,2]}-k\\
 1+p,1-\cB+p-\bfm-k,2-\cA+p
 \end{matrix}\right),
 \label{eq:vanishing-digamma-identity}
\end{align}
where for $q=p>0$, we have
\begin{align}
 C_k&=-\frac{(q-1)!\Gamma(\cB+\bfm+k)\Gamma(\cA-1)}
 {\Gamma(q+k+n_1+1)\Gamma(k+n_2+1)\Gamma(\cB)\Gamma(\cA+\bfn_{[1,2]}+k)},
 \label{eq:c-k-positive-p}\\
 \widehat C_k&=
 \frac{(-1)^{k+n_2}(q-k-n_2-1)!\Gamma(\cB-p+\bfm+k)\Gamma(\cA-p-1)}
 {q!\,\Gamma(k+n_1+1)\Gamma(\cB-p)\Gamma(\cA-p+\bfn_{[1,2]}+k)}
 {},
 \label{eq:hat-c-k-positive-p}
\end{align}
and for  $q=-p>0$,  we have
\begin{align*}
 C_k&=
 \frac{(-1)^{k+n_1}(q-k-n_1-1)!\Gamma(\cB+\bfm+k)\Gamma(\cA-1)}
 {q!\,\Gamma(k+n_2+1)\Gamma(\cB)\Gamma(\cA+\bfn_{[1,2]}+k)},
\\ 
 \widehat C_k&=-\frac{(q-1)!\Gamma(\cB-p+\bfm+k)\Gamma(\cA-p-1)}
 {\Gamma(q+k+n_2+1)\Gamma(k+n_1+1)\Gamma(\cB-p)\Gamma(\cA-p+\bfn_{[1,2]}+k)}.
\end{align*}
Note that hypergeometric series on the right-hand side of \eqref{eq:vanishing-digamma-identity} terminate before a negative integer denominator is reached. 
\end{corollary}
\begin{proof}
Indeed, take $p\ge1$.  The factor of the product $\cG_k\cH_{j}\Theta_{j}$ dependent on integer arguments (and hence potentially singular) has the form
$$
\frac{(-k-n_2)_{j}(-p-n_1-k)_{j}[\psi(1-p;j)-\psi(-p-k-n_1;j)]}{\Gamma(1-p)\Gamma(p+n_1+k+1)\Gamma(k+n_2+1)(1-p)_{j}j!}
$$
The summation range in $j$ is $0,\ldots,k+n_2$, where $k+n_2<p$  by assumption, so that $(1-p)_{j}\ne0$. Further, $j<p+k+n_1$, so that all digamma terms are non-singular in view of \eqref{eq:digamma-increment-negative} and the factor $1/\Gamma(1-p)=0$ annihilates the first sum in \eqref{eq:finite-digamma-sum}.  Similarly,  the factor of the product $\hat{\cG}_k\hat{\cH}_{j}\hat{\Theta}_{j}$ dependent on integer arguments (and hence potentially singular) has the form
$$
\frac{(-k-n_1)_{j}(p-n_2-k)_{j}[\psi(p-n_2-k;j)-\psi(1+p;j)]}{\Gamma(1+p)\Gamma(-p+k+n_2+1)\Gamma(k+n_1+1)(1+p)_{j}}.
$$
As $p-n_2-k>0$ and $p>0$ by assumption all digamma terms are non-singular, while  $1/\Gamma(-p+k+n_2+1)=0$ annihilates the second sum in \eqref{eq:finite-digamma-sum}.  Formulas \eqref{eq:c-k-positive-p} and \eqref{eq:hat-c-k-positive-p} are evaluations of $\cG_k\Psi_{k}$ and $\hat{\cG}_k\hat{\Psi}_{k}$ using \eqref{eq:psi-over-gamma-negative}.  The argument for negative $p$ is similar. 
\end{proof}

\section{Degeneration of a contiguous type sum-product identity}\label{sec:contiguous-degeneration}

The following relation was obtained in  \cite[eq. (6.7)]{Cetinkaya_Karp_2021} by taking the limit $q\to1$ of the identity \cite[Theorem~1.1]{GITZ_2015}:
\begin{multline}\label{eq:contiguous-product-identity}
a(b-d)(c-d)(b+c-a) {}_{r+4}F_{s+3}\left(\begin{array}{c}
 b+c-a-1,b+c-2,c,d,\bfe\\ a,b-1,b+c-d-1,\bff
\end{array} \middle| z\right)\\
\times {}_{r+4}F_{s+3}\left(\begin{array}{c}
 b+c-a+1,b+c,c,d,\bfe+1\\ a,b+1,b+c-d+1,\bff+1
\end{array} \middle| z\right)\\
-d(b-a)(c-a)(b+c-d){}_{r+4}F_{s+3}\left(\begin{array}{c}
 b+c-a,b+c-2,c,d-1,\bfe\\ a-1,b-1,b+c-d,\bff
\end{array} \middle| z\right)\\
\times {}_{r+4}F_{s+3}\left(\begin{array}{c}
 b+c-a,b+c,c,d+1,\bfe+1\\ a+1,b+1,b+c-d,\bff+1
\end{array} \middle| z\right)\\
=bc(a-d)(b+c-a-d) {}_{r+4}F_{s+3}\left(\begin{array}{c}
 b+c-a,b+c-2,c-1,d,\bfe\\ a-1,b,b+c-d-1,\bff
\end{array} \middle| z\right) \\
\times {}_{r+4}F_{s+3}\left(\begin{array}{c}
 b+c-a,b+c,c+1,d,\bfe+1\\ a+1,b,b+c-d+1,\bff+1
\end{array} \middle| z\right).
\end{multline}
This relation fails when $b=1$. Somewhat surprisingly, computing the limit $b\to1$ does not lead to a digamma series: instead, we obtain another hypergeometric sum-product identity tated in the following theorem which we believe to be new.

\begin{theorem}\label{th:contiguous-degenerate-identity}
Let $a,c,d\in \mathbb{C}$, $\bfe \in \mathbb{C}^{r}, \bff \in \mathbb{C}^{s}$ with $r\le{s}$ be such that no displayed denominator parameter is a non-positive integer. Then the following identity holds as a formal power series in $z$, and hence wherever the displayed series converge
\begin{align}\label{eq:contiguous-degenerate-identity}
& a(c\!-\!d)(1\!-\!d)(1\!+\!c\!-\!a)F\!\left(\!\!\begin{array}{c}
 c-a+2,1+c,c,d,\bfe+1\\ a,2,c-d+2,\bff+1
\end{array} \middle| z\right) \!F\!\left(\!\!\begin{array}{c}
 c-a,c-1,c-1,d,\bfe\\ 1,a,c-d,\bff
\end{array} \middle| z\right ) \notag \\
&-d(c-a)(1-a)(1+c-d) F\!\left(\!\!\begin{array}{c}
 c-a+1,c,c+1,d+1,\bfe+1\\ a+1,2,c-d+1,\bff+1
\end{array} \middle| z\right) \notag \\
& \times F\!\left(\!\!\begin{array}{c}
 c-a+1,c-1,c-1,d-1,\bfe\\ 1,a-1,c-d+1,\bff
\end{array} \middle| z\right) -c(a-d)(1+c-a-d) \notag \\
&\times  F\!\left(\!\!\begin{array}{c}
 1+c-a,c-1,c-1,d,\bfe\\ a-1,1,c-d,\bff
\end{array} \middle| z\right) F\!\left(\!\!\begin{array}{c}
 1+c-a,1+c,c+1,d,\bfe+1\\ a+1,1,c-d+2,\bff+1
\end{array} \middle| z\right) \notag \\
&=z \frac{dc(a-d)(1+c-a-d)(1+c-a)(c-1)^2(\bfe)_1}{(c-d)(1-a)(\bff)_1} \,F\!\left(\!\!\begin{array}{c}
 c-a+2,1+c,c,d,\bfe+1\\ a,2,c-d+2,\bff+1
\end{array} \middle| z\right) \notag\\
&  \times  F\!\left(\!\!\begin{array}{c} 
 c-a+1,c,c+1,d+1,\bfe+1\\ a+1,2,c-d+1,\bff+1
\end{array} \middle| z\right). 
\end{align} 
\end{theorem}
\begin{proof}
See the Appendix.
\end{proof}

We placed the proof of the above theorem in the Appendix because it is a particular case of a more general result containing an additional parameter $\kappa$ and established using contiguous relations below. We recover \eqref{eq:contiguous-degenerate-identity} when $\kappa=1$.

\begin{theorem}\label{th:contiguous-generalization}
Let $a,c,d,\kappa\in\mathbb{C}$, $\bfe \in \mathbb{C}^{r}, \bff \in \mathbb{C}^{s}$ with $r\le{s}$ be such that  no displayed denominator parameter is a non-positive integer. Define
\begin{align*}
&A\!=\!F\!\left(\!\!\begin{array}{c}
 c-a+\kappa,c+\kappa,c, d+1,\bfe+1\\ a+1,\kappa+1,c-d+\kappa,\bff+1
\end{array}\! \middle| z\!\right), ~ C\!=\! F\!\left(\!\!\begin{array}{c}
 c-a+\kappa,c+\kappa-2,c-1,d-1,\bfe\\ a-1,\kappa, c-d+\kappa,\bff
\end{array}\! \middle| z\!\right), \\
&B\!=\! F\!\left(\!\!\begin{array}{c}
 c-a+\kappa+1,c+\kappa,c,d,\bfe+1\\ a,\kappa+1,c-d+\kappa+1,\bff+1
\end{array}\! \middle| z\!\right), ~ D\!=\! F\!\left(\!\!\begin{array}{c}
 c-a+\kappa-1,c+\kappa-2,c-1,d,\bfe\\ a,\kappa,c-d+\kappa-1,\bff
\end{array} \!\middle| z\!\right),\\
&H\!=\! F\!\left(\!\!\begin{array}{c}
 c-a+\kappa,c+\kappa,c+1,d,\bfe+1\\ a+1,\kappa,c-d+\kappa+1,\bff+1
\end{array} \!\middle| z\!\right), ~ E\!=\!F\!\left(\!\!\begin{array}{c}
 c-a+\kappa,c+\kappa-2,c-1,d,\bfe\\ a-1,\kappa,c-d+\kappa-1,\bff
\end{array} \!\middle| z\!\right).
\end{align*} 
Then the following identity holds as a formal power series in $z$, and hence wherever the displayed series converge
\begin{align}\label{eq:contiguous-generalization}
& a(c-d)(\kappa-d)(\kappa+c-a)B D-d(c-a)(\kappa-a)(\kappa+c-d) AC \notag\\
& -\kappa c(a-d)(c-a-d+\kappa) EH \notag\\
&=z \frac{d(a-d)(c-a-d+\kappa)(c-1)(\kappa+c-a)(c+\kappa-1)(c+\kappa-2) (\bfe)_1}
{\kappa\,(\kappa+c-d-1)\,(1-a)\,(\bff)_1} AB .
\end{align}
\end{theorem}
\begin{proof} 
We claim that 
\begin{align}
&E-D =
\frac{(c-1)(\kappa+c-1)(\kappa+c-2) d (\bfe)_1}
{\kappa\,(\kappa+c-d-1)\,a(a-1)\,(\bff)_1}
\; zA \label{eq:e-minus-d-general}\\
&E-C =
\frac{(c-a+\kappa)(\kappa+c-1)(\kappa+c-2)(c-1) (\bfe)_1}
{\kappa\,(\kappa+c-d)(\kappa+c-d-1)\,(a-1)\,(\bff)_1}
\; zB. \label{eq:e-minus-c-general}
\end{align}
Indeed, simple manipulations with rising factorials yield
\[
\frac{(c-a+\kappa)_n}{(a-1)_n}
=
\frac{c-a+\kappa+n-1}{c-a+\kappa-1}
\cdot
\frac{a+n-1}{a-1}
\cdot
\frac{(c-a+\kappa-1)_n}{(a)_n}.
\]
Therefore, the coefficient at $z^n$ in $E-D$ for $n\ge1$ equals
\begin{align*}
&\frac{(c+\kappa-2)_n(c-1)_n(d)_n(\mathbf e)_n}{(\kappa)_n(c-d+\kappa-1)_n(\mathbf f)_n\,n!}
\left[
\frac{(c-a+\kappa)_n}{(a-1)_n}
-
\frac{(c-a+\kappa-1)_n}{(a)_n}
\right] \\
&=\frac{(c+\kappa-2)_n(c-1)_n(d)_n(\mathbf e)_n}{(\kappa)_n(c-d+\kappa-1)_n(\mathbf f)_n\,n!}\frac{(c-a+\kappa-1)_n}{(a)_n}
\left[
\frac{(c-a+\kappa+n-1)(a+n-1)}{(c-a+\kappa-1)(a-1)} - 1
\right]\\
&=
\frac{(c+\kappa+n-2)}{(c-a+\kappa-1)(a-1)}
\cdot
\frac{(c-a+\kappa-1)_n(c+\kappa-2)_n(c-1)_n(d)_n(\mathbf e)_n}
{(a)_n(\kappa)_n(c-d+\kappa-1)_n(\mathbf f)_n\,(n-1)!}.
\end{align*}
Shifting the summation index in view of 
$(x)_{n+1}=x(x+1)_{n}$, we obtain \eqref{eq:e-minus-d-general}. The same  argument gives  \eqref{eq:e-minus-c-general}. Next, we claim that 
\begin{multline}\label{eq:abh-general}
a(c-d)(\kappa-d)(\kappa+c-a)B-d(c-a)(\kappa-a)(\kappa+c-d) A  \\
= \kappa c(a-d)(\kappa+c-a-d) H
\end{multline}
or, writing $B=\sum_{n=0}^\infty B_n$, $A=\sum_{n=0}^\infty A_n$,  $H=\sum_{n=0}^\infty H_n$, an equivalent claim is
\begin{multline*}
\sum_{n=0}^\infty a(c-d)(\kappa-d)(\kappa+c-a)B_n-\sum_{n=0}^\infty d(c-a)(\kappa-a)(\kappa+c-d) A_n  \\
= \sum_{n=0}^\infty \kappa c(a-d)(\kappa+c-a-d) H_n.
\end{multline*}
 Expressing $A_n$ and $B_n$ in terms of $H_n$ yields
\begin{align*}
&B_n= H_n\,
\frac{\kappa+c-a+n}{\kappa+c-a}\,
\frac{c}{c+n}\,
\frac{a+n}{a}\,
\frac{\kappa}{\kappa+n}, \\
&A_n
=
H_n\,
\frac{c}{c+n}\,
\frac{d+n}{d}\,
\frac{\kappa}{\kappa+n}\,
\frac{\kappa+c-d+n}{\kappa+c-d}.
\end{align*}
Substituting these into the left-hand side of \eqref{eq:abh-general}, we get by an elementary calculation
\begin{align*}
&\text{LHS of \eqref{eq:abh-general}}
=
\sum_{n=0}^\infty \frac{c\kappa{H_n}}{(c+n)(\kappa+n)}
\Big\{
(c-d)(\kappa-d)(a+n)(\kappa+c-a+n)\\
&-(c-a)(\kappa-a)(d+n)(\kappa+c-d+n)\Big\}= \sum_{n=0}^\infty c\kappa(a-d)(\kappa+c-a-d)H_n
\end{align*}
confirming \eqref{eq:abh-general}.  Expressing $D$ from \eqref{eq:e-minus-d-general} and $C$ from \eqref{eq:e-minus-c-general} and substituting these expressions into the left-hand side of \eqref{eq:contiguous-generalization}, we obtain using \eqref{eq:abh-general}
\begin{align*}
&\text{LHS of \eqref{eq:contiguous-generalization}}\!=\!a(c-d)(\kappa-d)(\kappa+c-a)B\bigg[E-\frac{(c-1)(c+\kappa-1)(c+\kappa-2) d (\bfe)_1}
{\kappa\,(c-d+\kappa-1)\,a(a-1)\,(\bff)_1}
\; zA \bigg]\\
&-d(c-a)(\kappa-a)(c-d+\kappa) A \bigg[E-\frac{(c-a+\kappa)(c+\kappa-1)(c+\kappa-2)(c-1) (\bfe)_1}
{\kappa\,(c-d+\kappa)(c-d+\kappa-1)\,(a-1)\,(\bff)_1}
\; zB \bigg]\\
&=\big(a(c-d)(\kappa-d)(c-a+\kappa)B-d(c-a)(\kappa-a)(c-d+\kappa)A\big)E\\
&+\frac{(c-1)(c+\kappa-1)(c+\kappa-2) (\bfe)_1}
{\kappa\,(c-d+\kappa-1)\,(a-1)\,(\bff)_1} \Big[d(c-a)(\kappa-a)(c-a+\kappa)\\
&-d(c-d)(\kappa-d)(c-a+\kappa) \Big]zAB=\kappa c(a-d)(c-a-d+\kappa)EH\\
&-z\frac{d(a-d)(c-a-d+\kappa)(c-1)(c-a+\kappa)(c+\kappa-1)(c+\kappa-2) (\bfe)_1}
{\kappa\,(c-d+\kappa-1)\,(a-1)\,(\bff)_1}AB
\end{align*}
which is precisely \eqref{eq:contiguous-generalization}.
\end{proof}

Another identity of a similar flavor obtained in \cite[Lemma 6.5]{Cetinkaya_Karp_2021} by $q\to1$ limits of  \cite[Corollary~1.2]{GITZ_2015} is the following 
\begin{multline}\label{eq:two-index-contiguous-identity}
a_1(a_2-b_{1}){}_{r+1}F_{r}\left(\begin{array}{c}
 a_1-1,\bfa_{[1]}\\ b_1-1,\bfb_{[1]}
\end{array} \middle| z\right){}_{r+1}F_{r}\left(\begin{array}{c}
 a_2,\bfa_{[2]}+1\\ \bfb+1
\end{array} \middle| z\right)\\-a_2(a_1-b_1){}_{r+1}F_{r}\left(\begin{array}{c}
 a_2-1,\bfa_{[2]}\\ b_1-1,\bfb_{[1]}
\end{array} \middle| z\right) {}_{r+1}F_{r}\left(\begin{array}{c}
 a_1,\bfa_{[1]}+1\\ \bfb+1
\end{array} \middle| z\right)\\
=b_1(a_2-a_1){}_{r+1}F_{r}\left(\begin{array}{c}
 \bfa\\ \bfb
\end{array} \middle| z\right){}_{r+1}F_{r}\left(\begin{array}{c}
 a_1,a_2,\bfa_{[1,2]}+1\\ b_1,\bfb_{[1]}+1
\end{array} \middle| z\right).
\end{multline}

This identity fails when $b_1=1$ as the first two terms become singular. The limiting case of  \eqref{eq:two-index-contiguous-identity} as $b_1\to1$ somewhat surprisingly does not contain digamma functions. Instead, it has the form given in 

\begin{theorem}\label{th:two-index-degenerate-identity}
If $\bfa \in \mathbb{C}^{r+1}, \bfb_{[1]} \in \mathbb{C}^{r-1}$ with $r\ge1$, and no displayed denominator parameter is a non-positive integer, then the following identity holds 
as a formal power series in $z$, and hence wherever the displayed series converge
\begin{align}\label{eq:two-index-degenerate-identity}
&a_1(a_2-1) F\left(\begin{array}{c}
 a_2,\bfa_{[2]}+1\\ 2, \bfb_{[1]}+1
\end{array} \middle| z\right) F\left(\begin{array}{c} 
 a_1-1,\bfa_{[1]}\\ 1, \bfb_{[1]}
\end{array} \middle| z\right)   \notag\\
& -a_2(a_1-1) F\left(\begin{array}{c}
 a_1,\bfa_{[1]}+1\\ 2, \bfb_{[1]}+1
\end{array} \middle| z\right)  F\left(\begin{array}{c} 
 a_2-1,\bfa_{[2]}\\ 1, \bfb_{[1]}
\end{array} \middle| z\right) \notag\\
&-(a_2-a_1) F\left(\begin{array}{c}
 \bfa\\ 1,\bfb_{[1]}
\end{array} \middle| z\right) F\left(\begin{array}{c}
 a_1,a_2,\bfa_{[1,2]}+1\\ 1,\bfb_{[1]}+1
\end{array} \middle| z\right) \notag\\
&=z\frac{(\bfa)_1(a_1-a_2)}{(\bfb_{[1]})_1}    F\left(\begin{array}{c} 
 a_2,\bfa_{[2]}+1\\2, \bfb_{[1]}+1
\end{array} \middle| z\right)  F\left(\begin{array}{c}
 a_1,\bfa_{[1]}+1\\ 2, \bfb_{[1]}+1
\end{array} \middle| z\right). 
\end{align}
Here $(\bfa)_1=a_1a_2 \ldots a_{r+1}$ and $(\bfb_{[1]})_1=b_2 \ldots b_r$.
\end{theorem}

\begin{proof}
See the Appendix.
\end{proof}

The identity \eqref{eq:two-index-degenerate-identity} presented above corresponds to the special case $c=1$ of a more general result. Replacing the lower parameters $1$ and $2$ by generic  $c$ and $c+1$ suggests the following generalization which we state as \eqref{eq:two-index-generalization}. For a proof, we again employ contiguous relations.
\begin{theorem}\label{th:two-index-generalization}
Suppose $r\ge2$ and $s\ge0$ are integers.
Let $a_1,\dots,a_r,b_1,\dots,b_s,c\in\mathbb C$, with $c,b_1,\dots,b_s\notin\{0,-1,-2,\dots\}$. When $s=0$, parameter vectors and products indexed by the $b_j$'s are understood to be empty. Define
\begin{align*}
A&=F\left(\begin{array}{c}
 a_1,\bfa_{[1]}+1\\ c+1, \bfb +1
\end{array} \middle| z\right), ~~ B=F\left(\begin{array}{c}
 a_2,\bfa_{[2]}+1\\ c+1, \bfb +1
\end{array} \middle| z\right), ~~ 
C=F\left(\begin{array}{c} 
 a_2-1,\bfa_{[2]}\\ c, \bfb 
\end{array} \middle| z\right)
\\
D&=F\left(\begin{array}{c} 
 a_1-1,\bfa_{[1]}\\ c, \bfb 
\end{array} \middle| z\right),  \quad 
E=F\left(\begin{array}{c}
 \bfa\\ c,\bfb 
\end{array} \middle| z\right),\quad 
H=F\left(\begin{array}{c}
 a_1,a_2,\bfa_{[1,2]}+1\\ c,\bfb +1
\end{array} \middle| z\right).
\end{align*}
Then the following identity holds as a formal power series in $z$, and hence wherever the displayed series converge:
\begin{equation}\label{eq:two-index-generalization}
a_1(a_2-c)\,BD-a_2(a_1-c)\,A\,C
-
c(a_2-a_1)\,EH
=z\,\frac{a_1a_2\cdots a_r\,(a_1-a_2)}{c\,b_1\cdots b_s}\,AB.
\end{equation}
\end{theorem}
\begin{proof}
We first record two elementary contiguous identities. We claim that
\begin{equation}\label{eq:e-minus-d}
E-D=\frac{a_2a_3\cdots a_r}{c\,b_1b_2\cdots b_s}\,zA
\end{equation}
and
\begin{equation}\label{eq:e-minus-c}
E-C=\frac{a_1a_3\cdots a_r}{c\,b_1b_2\cdots b_s}\,zB.
\end{equation}
To prove \eqref{eq:e-minus-d}, \eqref{eq:e-minus-c}, compare the coefficients at $z^k$ on the left and right-hand sides, similarly to the proof of Theorem~\ref{th:contiguous-generalization}.
Next, we assert that 
\begin{equation}\label{eq:abh-relation}
a_1(a_2-c)\,B-a_2(a_1-c)\,A=c(a_2-a_1)\,H.
\end{equation}
Again, compare coefficients at powers of $z$. The coefficient at $z^k$ on the left-hand side is
\begin{align*}
&\frac{(a_1)_k(a_2)_k(a_3+1)_k\cdots(a_r+1)_k}
{(c+1)_k(b_1+1)_k\cdots(b_s+1)_k\,k!}
\Bigl[(a_2-c)(a_1+k)-(a_1-c)(a_2+k)\Bigr]\\
&=c(a_2-a_1)\,
\frac{(a_1)_k(a_2)_k(a_3+1)_k\cdots(a_r+1)_k}
{(c)_k(b_1+1)_k\cdots(b_s+1)_k\,k!},
\end{align*}
which coincides with the coefficient at $z^k$ of the function $c(a_2-a_1)H$ which proves \eqref{eq:abh-relation}.
Replacing $D$ and $C$ in \eqref{eq:two-index-generalization} by their expressions from \eqref{eq:e-minus-d} and \eqref{eq:e-minus-c}, respectively,  we rewrite
the left-hand side of \eqref{eq:two-index-generalization} as follows
\begin{multline*}
a_1(a_2-c)BD-a_2(a_1-c)AC 
\\=
a_1(a_2-c)B\left(E-\frac{a_2\cdots a_r}{c\,b_1\cdots b_s}zA\right)
-a_2(a_1-c)A\left(E-\frac{a_1a_3\cdots a_r}{c\,b_1\cdots b_s}zB\right)
\\=
\bigl(a_1(a_2-c)B-a_2(a_1-c)A\bigr)E +z\,\frac{a_1a_2\cdots a_r}{c\,b_1\cdots b_s}
\bigl((a_1-c)-(a_2-c)\bigr)AB \\
=
\bigl(a_1(a_2-c)B-a_2(a_1-c)A\bigr)E
+z\,\frac{a_1a_2\cdots a_r\,(a_1-a_2)}{c\,b_1\cdots b_s}AB.
\end{multline*}
Now substitute \eqref{eq:abh-relation}; this yields exactly \eqref{eq:two-index-generalization}.
\end{proof}

\begin{corollary}
If $r=2$ and $s=0$, then  identity \eqref{eq:two-index-generalization} becomes
\begin{align*}
&a_1(a_2-c) {}_2F_{1}\left(\begin{array}{c}
 a_2,a_1+1\\ c+1
\end{array} \middle| z\right) {}_2F_{1}\left(\begin{array}{c} 
 a_1-1,a_2\\ c
\end{array} \middle| z\right)   \notag\\
& -a_2(a_1-c) {}_2F_{1}\left(\begin{array}{c}
 a_1,a_2+1\\ c+1
\end{array} \middle| z\right)  {}_2F_{1}\left(\begin{array}{c} 
 a_2-1,a_1\\ c
\end{array} \middle| z\right)-c(a_2-a_1) \bigg[{}_2F_{1}\left(\begin{array}{c}
 a_1,a_2\\ c
\end{array} \middle| z\right)\bigg]^2 \notag\\
&= \frac{za_1 a_2(a_1-a_2)}{c}  {}_2F_{1}\left(\begin{array}{c} 
 a_2,a_1+1\\ c+1
\end{array} \middle| z\right)  {}_2F_{1}\left(\begin{array}{c}
 a_1,a_2+1\\ c+1
\end{array} \middle| z\right).  \notag
\end{align*}
\end{corollary}
We note a curious identity for a sum of products of the Gauss functions ${}_2F_1$  discovered recently in \cite{Aptekarev_Dyachenko_Lysov_2025} somewhat similar in flavor to the above corollary.  It played a prominent role in establishing the perfectness of the Meixner--Sorokin system of weights.

\section*{Acknowledgment}

S.K.: This work was supported by the Moscow Center of Fundamental and Applied Mathematics, Agreement with the Ministry of Science and Higher Education of the Russian Federation, No. 075-15-2025-346.

\bibliographystyle{siam} 
\bibliography{references}

\phantomsection
\section*{Appendix}
\addcontentsline{toc}{section}{Appendix}

\begin{proof}[Proof of Corollary~\ref{cr:euler-digamma-identity}]
Set $p=0$, $\bfm=\bfn=\mathbf0$, and $r=2$ in
Theorem~\ref{th:duality-degeneration}. Then $e=-1$, the sum involving $\Delta_j$ is empty,
$V_3=0$, and $\cA$ is the empty vector. Put
\begin{align*}
F_+(z)&={}_2F_1\!\left(\begin{array}{c}b_1,b_2\\1\end{array}\middle|z\right),
F_-(z)={}_2F_1\!\left(\begin{array}{c}1-b_1,1-b_2\\1\end{array}\middle|z\right),\\
S_+(z)&=\sum_{k=0}^\infty\frac{(b_1)_k(b_2)_k}{(k!)^2}
 \bigl(2\psi(1+k)-\psi(b_1+k)-\psi(b_2+k)\bigr)z^k,\\
S_-(z)&=\sum_{k=0}^\infty\frac{(1-b_1)_k(1-b_2)_k}{(k!)^2}
 \bigl(\psi(1-b_1+k)+\psi(1-b_2+k)-2\psi(1+k)\bigr)z^k.
\end{align*}
Substitution of the four functions $\Psi_1,\ldots,\Psi_4$ from
Theorem~\ref{th:duality-degeneration}, followed by \(\Gamma(x+k)=\Gamma(x)(x)_k\), reduces
\eqref{eq:duality-degeneration} to
\[
 F_+(z)S_-(z)+F_-(z)S_+(z)
 =\pi\bigl(\cot\pi b_1+\cot\pi b_2\bigr)F_+(z)F_-(z).
\]
Euler's transformation gives
$F_-(z)=(1-z)^{b_1+b_2-1}F_+(z)$. Dividing the last identity by
$F_-(z)$ yields \eqref{eq:euler-digamma-identity}; the result then extends across
removable zeros by analytic continuation. \end{proof}

\bigskip

\begin{proof}[Proof of Theorem~\ref{th:contiguous-degenerate-identity}] The identity \eqref{eq:contiguous-product-identity} is valid for whenever $b \neq 1$. Our goal is to compute the limit as $b\to1$. To this end, substituting $b =1+\epsilon$ in \eqref{eq:contiguous-product-identity} and using \eqref{eq:regularized-hypergeometric}, we obtain
\begin{align}\label{eq:contiguous-identity-regularized}
 &\frac{a(1+\epsilon-d)(c-d)(1+\epsilon+c-a)\Gamma(a)\Gamma(\epsilon)\Gamma(\epsilon+c-d)\Gamma(a)\Gamma(\epsilon+2)\Gamma(\epsilon+c-d+2)}{\Gamma(\epsilon+c-a)\Gamma(\epsilon+c-1)\Gamma(c)\Gamma(d)\Gamma(\epsilon+c-a+2)\Gamma(\epsilon+c+1)\Gamma(c)\Gamma(d)} \notag\\
& \frac{\Gamma(\bff)\Gamma(\bff+1)}{\Gamma(\bfe)\Gamma(\bfe+1)} \phi\!\left(\!\!\begin{array}{c}
 \epsilon+c-a,\epsilon+c-1,c,d,\bfe\\ a,\epsilon,\epsilon+c-d,\bff
\end{array} \middle| z\!\right) \phi\!\left(\!\!\begin{array}{c}
 \epsilon+c-a+2,1+\epsilon+c,c,d,\bfe+1\\ a,\epsilon+2,\epsilon+c-d+2,\bff+1
\end{array} \middle| z\!\right) \notag\\
&-\frac{d(1+\epsilon-a)(c-a)(1+\epsilon+c-d)\Gamma(\epsilon)\Gamma(1+\epsilon+c-d)\Gamma(\epsilon+2)\Gamma(1+\epsilon+c-d)}{\Gamma(1+\epsilon+c-a)\Gamma(\epsilon+c-1)\Gamma(c)\Gamma(1+\epsilon+c-a)\Gamma(1+\epsilon+c)\Gamma(c)}\notag \\
&\frac{\Gamma(a-1)\Gamma(a+1)\Gamma(\bff)\Gamma(\bff+1)}{\Gamma(d+1)\Gamma(d-1)\Gamma(\bfe)\Gamma(\bfe+1)}\phi\left(\begin{array}{c}
 1+\epsilon+c-a,\epsilon+c-1,c,d-1,\bfe\\ a-1,\epsilon,1+\epsilon+c-d,\bff
\end{array} \middle| z\right) \notag\\
&\phi\left(\begin{array}{c}
 1+\epsilon+c-a,1+\epsilon+c,c,d+1,\bfe+1\\ a+1,\epsilon+2,1+\epsilon+c-d,\bff+1
\end{array} \middle| z\right)=c(1+\epsilon)(a-d)(1+\epsilon+c-a-d)  \notag\\
&F\!\left(\!\!\begin{array}{c}
 1+\epsilon+c-a,\epsilon+c-1,c-1,d,\bfe\\ a-1,1+\epsilon,\epsilon+c-d,\bff
\end{array} \middle| z\!\right) F\!\left(\!\!\begin{array}{c}
 1+\epsilon+c-a,1+\epsilon+c,c+1,d,\bfe+1\\ a+1,1+\epsilon,\epsilon+c-d+2,\bff+1
\end{array} \middle| z\!\right).  
\end{align}
As above, we can define
\begin{align*}
&f_1(\epsilon)=\Gamma(\epsilon) h_1(\epsilon)=\Gamma(\epsilon)g_1(\epsilon)\phi\left(\begin{array}{c}
 \epsilon+c-a,\epsilon+c-1,c,d,\bfe\\ a,\epsilon,\epsilon+c-d,\bff
\end{array} \middle| z\right) \\
&\hspace{3cm}\times\phi\left(\begin{array}{c}
 \epsilon+c-a+2,1+\epsilon+c,c,d,\bfe+1\\ a,\epsilon+2,\epsilon+c-d+2,\bff+1
\end{array} \middle| z\right),\\
&f_2(\epsilon)=\Gamma(\epsilon) h_2(\epsilon)=\Gamma(\epsilon)g_2(\epsilon)\phi\left(\begin{array}{c}
 1+\epsilon+c-a,\epsilon+c-1,c,d-1,\bfe\\ a-1,\epsilon,1+\epsilon+c-d,\bff
\end{array} \middle| z\right) \\
&\hspace{3cm}\times\phi\left(\begin{array}{c}
 1+\epsilon+c-a,1+\epsilon+c,c,d+1,\bfe+1\\ a+1,\epsilon+2,1+\epsilon+c-d,\bff+1
\end{array} \middle| z\right),
\end{align*}
where
\begin{align*}
&g_1(\epsilon)= \\
&\frac{a(c-d)\Gamma^2(a)\Gamma(\bff)\Gamma(\bff+1)(1+\epsilon-d)(1+\epsilon+c-a)\Gamma(\epsilon+c-d)\Gamma(\epsilon+2)\Gamma(\epsilon+c-d+2)}{\Gamma^2(c)\Gamma^2(d)\Gamma(\bfe)\Gamma(\bfe+1)\Gamma(\epsilon+c-a)\Gamma(\epsilon+c-1)\Gamma(\epsilon+c-a+2)\Gamma(\epsilon+c+1)},\\
&g_2(\epsilon)= \\
&\frac{d(c-a)\Gamma(a-1)\Gamma(a+1)\Gamma(\bff)\Gamma(\bff+1)(1+\epsilon-a)(1+\epsilon+c-d)\Gamma(\epsilon+2)\Gamma^2(1+\epsilon+c-d)}{\Gamma^2(c)\Gamma(d-1)\Gamma(d+1)\Gamma(\bfe)\Gamma(\bfe+1)\Gamma(\epsilon+c-1)\Gamma(1+\epsilon+c)\Gamma^2(1+\epsilon+c-a)}.
\end{align*} 
Again, employing Taylor expansion for $h_1(\epsilon)$ and $h_2(\epsilon)$, we obtain
\begin{align}\label{eq:contiguous-degeneration-limit}
&f_1(\epsilon)-f_2(\epsilon)= \Gamma(\epsilon) [h_1(\epsilon)-h_2(\epsilon)]= \Gamma(\epsilon) [h_1(0)+\epsilon h'_1(0)-h_2(0)-\epsilon h'_2(0)+O(\epsilon^2)] \notag\\
& = \epsilon \Gamma(\epsilon) [h'_1(0)-h'_2(0)+O(\epsilon)] = \Gamma(\epsilon+1) [h'_1(0)-h'_2(0)+O(\epsilon)] \notag\\
&\rightarrow [h'_1(0)-h'_2(0)] ~~{\rm as}~~ \epsilon \rightarrow 0,
\end{align}
where the last equality follows because $h_2(0)=h_1(0)$ which can be established by shifting the index of summation similarly to the proof of Theorem~\ref{th:duality-degeneration}. For $h'_1(0)$, we use the relation \eqref{eq:reciprocal-gamma-derivative}. Thus, we get
\begin{align*}
&h'_1(0) =\left[ \frac{\partial}{\partial\epsilon} h_1(\epsilon)  \right]_{\epsilon = 0}= g'_1(0)\phi\!\left(\!\!\begin{array}{c}
 c-a+2,1+c,c,d,\bfe+1\\ a,2,c-d+2,\bff+1
\end{array} \middle| z\!\right)\phi\!\left(\!\!\begin{array}{c}
 c-a,c-1,c,d,\bfe\\ a,0,c-d,\bff
\end{array} \middle| z\!\right) \\
& +g_1(0)\phi\!\left(\!\!\begin{array}{c}
 c-a+2,1+c,c,d,\bfe+1\\ a,2,c-d+2,\bff+1
\end{array} \middle| z\!\right)\bigg\{ \frac{\Gamma(c-a)\Gamma(c-1)\Gamma(c)\Gamma(d)\Gamma(\bfe)}{\Gamma(a)\Gamma(c-d)\Gamma(\bff)}\\
& + \sum_{k=1}^{\infty} \frac{\Gamma(c-a+k)\Gamma(c-1+k)\Gamma(c+k)\Gamma(d+k)\Gamma(\bfe+k) z^k}{\Gamma(a+k)\Gamma(k)\Gamma(c-d+k)\Gamma(\bff+k)k!}  \\
& \times\big[\psi(c\!-\!a\!+\!k)+\psi(c\!-\!1\!+\!k)-\psi(k)-\psi(c\!-\!d\!+\!k)\big]\bigg\}+g_1(0)\phi\!\left(\!\!\begin{array}{c}
 c-a,c-1,c,d,\bfe\\ a,0,c-d,\bff
\end{array} \middle| z\!\right)\\
&\times\sum_{k=0}^{\infty} \frac{\Gamma(c-a+2+k)\Gamma(c+1+k)\Gamma(c+k)\Gamma(d+k)\Gamma(\bfe+1+k)z^k}{\Gamma(a+k)\Gamma(2+k)\Gamma(c-d+2+k)\Gamma(\bff+1+k)k!} \\
& \times\big[\psi(c-a+k+2)+\psi(c-1+k+2)-\psi(k+2)-\psi(c-d+k+2)\big],
\end{align*} 
where $g_1:=g_1(0)$ and $g'_1(0)=g_1\widehat{g}_1$, with $g_1$ and $\widehat{g}_1$ given by \eqref{eq:g1-value} and \eqref{eq:hat-g1-value}, respectively.
\begin{align}
&g_1 = \frac{a(c-d)\Gamma^2(a)\Gamma(\bff)\Gamma(\bff+1)(1-d)(1+c-a)\Gamma(c-d)\Gamma(c-d+2)}{\Gamma^2(c)\Gamma^2(d)\Gamma(\bfe)\Gamma(\bfe+1)\Gamma(c-a)\Gamma(c-1)\Gamma(c-a+2)\Gamma(c+1)} \label{eq:g1-value},\\
&\widehat{g}_1=\frac{1}{1-d} + \frac{1}{1+c-a}+\psi(c-d)+\psi(2)+\psi(c-d+2) \notag\\
&\hspace{4cm }-\psi(c-a)-\psi(c-1)-\psi(c-a+2)-\psi(c+1)  \label{eq:hat-g1-value}.
\end{align}

Then we compute $h'_2(0)$ as follows:
\begin{align*}
&h'_2(0) = \left[ \frac{\partial}{\partial\epsilon} h_2(\epsilon)  \right]_{\epsilon = 0}=g'_2(0)\phi\left(\begin{array}{c}
 1+c-a,1+c,c,d+1,\bfe+1\\ a+1,2,1+c-d,\bff+1
\end{array} \middle| z\right) \times \\
&\phi\left(\begin{array}{c}
 1+c-a,c-1,c,d-1,\bfe\\ a-1,0,1+c-d,\bff
\end{array} \middle| z\right)+g_2(0)\phi\left(\begin{array}{c}
 1+c-a,1+c,c,d+1,\bfe+1\\ a+1,2,1+c-d,\bff+1
\end{array} \middle| z\right)\\
&\bigg[ \frac{\Gamma(c-a+1)\Gamma(c-1)\Gamma(c)\Gamma(d-1)\Gamma(\bfe)}{\Gamma(a-1)\Gamma(c-d+1)\Gamma(\bff)}+\\
&\sum_{k=1}^{\infty} \frac{\Gamma(1+c-a+k)\Gamma(c-1+k)\Gamma(c+k)\Gamma(d-1+k)\Gamma(\bfe+k) z^k}{\Gamma(a-1+k)\Gamma(k)\Gamma(1+c-d+k)\Gamma(\bff+k)k!}  \\
& \times\big[\psi(1+c-a+k)+\psi(c-1+k)-\psi(k)-\psi(1+c-d+k)\big]\bigg]\\
&+g_2(0)\phi\left(\begin{array}{c}
 1+c-a,c-1,c,d-1,\bfe\\ a-1,0,1+c-d,\bff
\end{array} \middle| z\right)\\
&\sum_{k=0}^{\infty} \frac{\Gamma(1+c-a+k)\Gamma(1+c+k)\Gamma(c+k)\Gamma(d+1+k)\Gamma(\bfe+1+k) z^k}{\Gamma(a+1+k)\Gamma(2+k)\Gamma(1+c-d+k)\Gamma(\bff+1+k)k!} \\
& \times\big[\psi(1+c-a+k)+\psi(c+1+k)-\psi(k+2)-\psi(1+c-d+k)\big],
\end{align*}
where $g_2:=g_2(0)$ and $g'_2(0)=g_2\widehat{g}_2$, with $g_2$ and $\widehat{g}_2$ given by \eqref{eq:g2-value} and \eqref{eq:hat-g2-value}, respectively.
\begin{align}
&g_2 = \frac{d(c-a)\Gamma(a-1)\Gamma(a+1)\Gamma(\bff)\Gamma(\bff+1)(1-a)(1+c-d)\Gamma^2(1+c-d)}{\Gamma^2(c)\Gamma(d-1)\Gamma(d+1)\Gamma(\bfe)\Gamma(\bfe+1)\Gamma^2(1+c-a)\Gamma(c-1)\Gamma(1+c)} ,\label{eq:g2-value} \\
&\widehat{g}_2 = \frac{1}{1-a} + \frac{1}{1+c-d}+ 2\psi(1+c-d) + \psi(2)- \psi(c-1)- \psi(c+1)- 2\psi(1+c-a) \label{eq:hat-g2-value}.
\end{align}
Next, we calculate 
\begin{align*}
&h'_1(0) -h'_2(0) = g_1 \Bigg[\widehat{g}_1 \phi\left(\begin{array}{c}
 c-a+2,1+c,c,d,\bfe+1\\ a,2,c-d+2,\bff+1
\end{array} \middle| z\right)\phi\left(\begin{array}{c}
 c-a,c-1,c,d,\bfe\\ a,0,c-d,\bff
\end{array} \middle| z\right) \\
& +\phi\left(\begin{array}{c}
 c-a+2,1+c,c,d,\bfe+1\\ a,2,c-d+2,\bff+1
\end{array} \middle| z\right)\bigg[ \frac{\Gamma(c-a)\Gamma(c-1)\Gamma(c)\Gamma(d)\Gamma(\bfe)}{\Gamma(a)\Gamma(c-d)\Gamma(\bff)}+\sum_{k=1}^{\infty} \cU_1(k) z^k \bigg]\\
&  +\phi\left(\begin{array}{c}
 c-a,c-1,c,d,\bfe\\ a,0,c-d,\bff
\end{array} \middle| z\right)\sum_{k=0}^{\infty} \cV_1(k) z^k\Bigg]\\
&-g_2 \Bigg[\widehat{g}_2 \phi\left(\begin{array}{c}
 1+c-a,1+c,c,d+1,\bfe+1\\ a+1,2,1+c-d,\bff+1
\end{array} \middle| z\right)\phi\left(\begin{array}{c}
 1+c-a,c-1,c,d-1,\bfe\\ a-1,0,1+c-d,\bff
\end{array} \middle| z\right) \\
&+\phi\left(\begin{array}{c}
 1+c-a,1+c,c,d+1,\bfe+1\\ a+1,2,1+c-d,\bff+1
\end{array} \middle| z\right) \bigg[ \frac{\Gamma(c-a+1)\Gamma(c-1)\Gamma(c)\Gamma(d-1)\Gamma(\bfe)}{\Gamma(a-1)\Gamma(c-d+1)\Gamma(\bff)}+\\
&\sum_{k=1}^{\infty} \cU_2(k) z^k \bigg]+\phi\left(\begin{array}{c}
 1+c-a,c-1,c,d-1,\bfe\\ a-1,0,1+c-d,\bff
\end{array} \middle| z\right)\sum_{k=0}^{\infty} \cV_2(k) z^k \Bigg],
\end{align*}
Here $\cU_1,\cV_1$ correspond to the two differentiated hypergeometric factors in $h_1$, while $\cU_2,\cV_2$ correspond to those in $h_2$; explicitly,
\begin{align*}
&\cU_1(k) = \frac{\Gamma(c-a+k)\Gamma(c-1+k)\Gamma(c+k)\Gamma(d+k)\Gamma(\bfe+k) }{\Gamma(a+k)\Gamma(k)\Gamma(c-d+k)\Gamma(\bff+k)k!}  \notag \\
&\hspace{4cm }\times\big[\psi(c-a+k)+\psi(c-1+k)-\psi(k)-\psi(c-d+k)\big], \\
&\cV_1(k)= \frac{\Gamma(c-a+2+k)\Gamma(c+1+k)\Gamma(c+k)\Gamma(d+k)\Gamma(\bfe+1+k) }{\Gamma(a+k)\Gamma(2+k)\Gamma(c-d+2+k)\Gamma(\bff+1+k)k!} \notag \\
&\hspace{2cm }\times\big[\psi(c-a+k+2)+\psi(c-1+k+2)-\psi(k+2)-\psi(c-d+k+2)\big], \\
&\cU_2(k) = \frac{\Gamma(1+c-a+k)\Gamma(c-1+k)\Gamma(c+k)\Gamma(d-1+k)\Gamma(\bfe+k) }{\Gamma(a-1+k)\Gamma(k)\Gamma(1+c-d+k)\Gamma(\bff+k)k!} \notag \\
&\hspace{3cm }\times \big[\psi(1+c-a+k)+\psi(c-1+k)-\psi(k)-\psi(1+c-d+k)\big], \\
&\cV_2(k) =  \frac{\Gamma(1+c-a+k)\Gamma(1+c+k)\Gamma(c+k)\Gamma(d+1+k)\Gamma(\bfe+1+k)}{\Gamma(a+1+k)\Gamma(2+k)\Gamma(1+c-d+k)\Gamma(\bff+1+k)k!} \notag\\
&\hspace{2cm }\times\big[\psi(1+c-a+k)+\psi(c+1+k)-\psi(k+2)-\psi(1+c-d+k)\big]. 
\end{align*}
Comparing \eqref{eq:g1-value} and \eqref{eq:g2-value}, we see that $g_1=g_2$; we denote their common value by $g$. Upon substituting the expression for $h'_1(0)$ and $h'_2(0)$ in \eqref{eq:contiguous-degeneration-limit} and then finally in \eqref{eq:contiguous-identity-regularized}, we obtain
\begin{align}\label{eq:contiguous-identity-intermediate}
& g\phi\!\left(\!\!\begin{array}{c}
 c-a+2,1+c,c,d,\bfe+1\\ a,2,c-d+2,\bff+1
\end{array} \middle| z\!\right)\sum_{k=1}^{\infty} \cU_1(k) z^k+ g\phi\!\left(\!\!\begin{array}{c}
 c-a,c-1,c,d,\bfe\\ a,0,c-d,\bff
\end{array} \middle| z\!\right)\sum_{k=0}^{\infty} \cV_1(k) z^k \notag \\
&-g\phi\!\left(\!\!\begin{array}{c}
 1+c-a,1+c,c,d+1,\bfe+1\\ a+1,2,1+c-d,\bff+1
\end{array} \middle| z\!\right)\sum_{k=1}^{\infty} \cU_2(k) z^k \notag\\
&-g\phi\left(\begin{array}{c}
 1+c-a,c-1,c,d-1,\bfe\\ a-1,0,1+c-d,\bff
\end{array} \middle| z\right)\sum_{k=0}^{\infty} \cV_2(k) z^k=c(a-d)(1+c-a-d) \notag\\
& F\left(\begin{array}{c}
 1+c-a,c-1,c-1,d,\bfe\\ a-1,1,1+c-d-1,\bff
\end{array} \middle| z\right)F\left(\begin{array}{c}
 1+c-a,1+c,c+1,d,\bfe+1\\ a+1,1,1+c-d+1,\bff+1
\end{array} \middle| z\right) \notag\\
&-g \widehat{g}_1  \,\phi\left(\begin{array}{c}
 c-a+2,1+c,c,d,\bfe+1\\ a,2,c-d+2,\bff+1
\end{array} \middle| z\right)\phi\left(\begin{array}{c}
 c-a,c-1,c,d,\bfe\\ a,0,c-d,\bff
\end{array} \middle| z\right) \notag\\
&+g\widehat{g}_2 \, \phi\left(\begin{array}{c}
 1+c-a,1+c,c,d+1,\bfe+1\\ a+1,2,1+c-d,\bff+1
\end{array} \middle| z\right)\phi\left(\begin{array}{c}
 1+c-a,c-1,c,d-1,\bfe\\ a-1,0,1+c-d,\bff
\end{array} \middle| z\right) \notag\\
&-a(c-d)(1-d)(1+c-a) F\left(\begin{array}{c}
 c-a+2,1+c,c,d,\bfe+1\\ a,2,c-d+2,\bff+1
\end{array} \middle| z\right) \notag\\
&+d(c-a)(1-a)(1+c-d) F\left(\begin{array}{c}
 1+c-a,1+c,c,d+1,\bfe+1\\ a+1,2,1+c-d,\bff+1
\end{array} \middle| z\right).
\end{align} 
Next, it is easy to check that 
\begin{align*}
&\phi\left(\begin{array}{c}
 c-a,c-1,c,d,\bfe\\ a,0,c-d,\bff
\end{array} \middle| z\right) = z\phi\left(\begin{array}{c}
 c-a+1,c,c+1,d+1,\bfe+1\\ a+1,2,c-d+1,\bff+1
\end{array} \middle| z\right)\\
&\phi\left(\begin{array}{c}
 c-a+1,c-1,c,d-1,\bfe\\ a-1,0,c-d+1,\bff
\end{array} \middle| z\right)=z \phi\left(\begin{array}{c}
 c-a+2,c,c+1,d,\bfe+1\\ a,2,c-d+2,\bff+1
\end{array} \middle| z\right)\\
&\sum_{k=1}^{\infty} \cU_1(k) z^k = \sum_{k=0}^{\infty} \cU_1(k+1) z^{k+1}, \quad \sum_{k=1}^{\infty} \cU_2(k) z^k=\sum_{k=0}^{\infty} \cU_2(k+1) z^{k+1}.
\end{align*}
Substituting these into \eqref{eq:contiguous-identity-intermediate}, we obtain
\begin{align*}
& g\phi\left(\begin{array}{c}
 c-a+2,1+c,c,d,\bfe+1\\ a,2,c-d+2,\bff+1
\end{array} \middle| z\right)\sum_{k=0}^{\infty} \cU_1(k+1) z^{k+1}\\
&+ g\phi\left(\begin{array}{c}
 c-a+1,c,c+1,d+1,\bfe+1\\ a+1,2,c-d+1,\bff+1
\end{array} \middle| z\right)\sum_{k=0}^{\infty} \cV_1(k) z^{k+1} \notag \\
&-g \phi\left(\begin{array}{c}
 1+c-a,1+c,c,d+1,\bfe+1\\ a+1,2,1+c-d,\bff+1
\end{array} \middle| z\right)\sum_{k=0}^{\infty} \cU_2(k+1) z^{k+1} \notag\\
&-g\phi\left(\begin{array}{c}
 c-a+2,c,c+1,d,\bfe+1\\ a,2,c-d+2,\bff+1
\end{array} \middle| z\right)\sum_{k=0}^{\infty} \cV_2(k) z^{k+1}=c(a-d)(1+c-a-d) \notag\\
& F\left(\begin{array}{c}
 1+c-a,c-1,c-1,d,\bfe\\ a-1,1,c-d,\bff
\end{array} \middle| z\right)F\left(\begin{array}{c}
 1+c-a,1+c,c+1,d,\bfe+1\\ a+1,1,c-d+2,\bff+1
\end{array} \middle| z\right) \notag\\
&-g \widehat{g}_1  \,\phi\left(\begin{array}{c}
 c-a+2,1+c,c,d,\bfe+1\\ a,2,c-d+2,\bff+1
\end{array} \middle| z\right)z \phi\left(\begin{array}{c}
 c-a+1,c,c+1,d+1,\bfe+1\\ a+1,2,c-d+1,\bff+1
\end{array} \middle| z\right) \notag\\
&+g\widehat{g}_2 \, \phi\left(\begin{array}{c}
 1+c-a,1+c,c,d+1,\bfe+1\\ a+1,2,1+c-d,\bff+1
\end{array} \middle| z\right)z \phi\left(\begin{array}{c}
 c-a+2,c,c+1,d,\bfe+1\\ a,2,c-d+2,\bff+1
\end{array} \middle| z\right) \notag\\
&-a(c-d)(1-d)(1+c-a) F\left(\begin{array}{c}
 c-a+2,1+c,c,d,\bfe+1\\ a,2,c-d+2,\bff+1
\end{array} \middle| z\right) \notag\\
&+d(c-a)(1-a)(1+c-d) F\left(\begin{array}{c}
 1+c-a,1+c,c,d+1,\bfe+1\\ a+1,2,1+c-d,\bff+1
\end{array} \middle| z\right).
\end{align*} 
This implies
\begin{align*}
& g\phi\left(\begin{array}{c}
 c-a+2,1+c,c,d,\bfe+1\\ a,2,c-d+2,\bff+1
\end{array} \middle| z\right)\sum_{k=0}^{\infty} [\cU_1(k+1)-\cV_2(k)] z^{k+1}\\
&+ g\phi\left(\begin{array}{c}
 c-a+1,c,c+1,d+1,\bfe+1\\ a+1,2,c-d+1,\bff+1
\end{array} \middle| z\right)\sum_{k=0}^{\infty} [\cV_1(k)-\cU_2(k+1)] z^{k+1}=c(a-d) \notag \\
&(1+c-a-d) F\!\left(\!\!\begin{array}{c}
 1+c-a,c-1,c-1,d,\bfe\\ a-1,1,c-d,\bff
\end{array} \middle| z\!\right)F\!\left(\!\!\begin{array}{c}
 1+c-a,1+c,c+1,d,\bfe+1\\ a+1,1,c-d+2,\bff+1
\end{array} \middle| z\!\right) \notag\\
&+zg (\widehat{g}_2-\widehat{g}_1)  \,\phi\left(\begin{array}{c}
 c-a+2,1+c,c,d,\bfe+1\\ a,2,c-d+2,\bff+1
\end{array} \middle| z\right) \phi\left(\begin{array}{c}
 c-a+1,c,c+1,d+1,\bfe+1\\ a+1,2,c-d+1,\bff+1
\end{array} \middle| z\right) \notag\\
&-a(c-d)(1-d)(1+c-a) F\!\left(\!\!\begin{array}{c}
 c-a+2,1+c,c,d,\bfe+1\\ a,2,c-d+2,\bff+1
\end{array} \middle| z\right) \notag\\
&+d(c-a)(1-a)(1+c-d) F\left(\begin{array}{c}
 1+c-a,1+c,c,d+1,\bfe+1\\ a+1,2,1+c-d,\bff+1
\end{array} \middle| z\right).
\end{align*} 
We proceed with some simplifications. First, 
\begin{align*}
\cU_1(k+1)-\cV_2(k)&\!=\!\frac{\Gamma(1+c-a+k)\Gamma(1+c+k)\Gamma(c+k)\Gamma(d+1+k)\Gamma(\bfe+1+k)}{\Gamma(a+1+k)\Gamma(2+k)\Gamma(1+c-d+k)\Gamma(\bff+1+k)k!} \notag\\
&\times\big[\psi(c+k)-\psi(c+1+k)+\psi(k+2)-\psi(k+1)\big]\\
&= \frac{(c-1)\Gamma(1+c-a+k)\Gamma^2(c+k)\Gamma(d+1+k)\Gamma(\bfe+1+k)}{\Gamma(a+1+k)\Gamma(2+k)\Gamma(1+c-d+k)\Gamma(\bff+1+k)(k+1)!}.
\end{align*}
Similarly, 
\begin{equation*}
\cV_1(k)-\cU_2(k+1)= - \frac{(c-1)\Gamma(c-a+2+k)\Gamma^2(c+k)\Gamma(d+k)\Gamma(\bfe+1+k)}{\Gamma(a+k)\Gamma(2+k)\Gamma(c-d+2+k)\Gamma(\bff+1+k)(k+1)!}.
\end{equation*}
From definitions \eqref{eq:hat-g1-value},\eqref{eq:hat-g2-value} and in view of the recurrence $\psi(x+1)=\psi(x)+1/x$, we deduce that the coefficient $\widehat{g}_2-\widehat{g}_1$ simplifies as follows: 
\begin{align*}
&\widehat{g}_2-\widehat{g}_1=\frac{1}{1-a} + \frac{1}{1+c-d}-\frac{1}{1-d} - \frac{1}{1+c-a}
\\
&
-\psi(c-d)+2\psi(c-d+1)-\psi(c-d+2)+\psi(c-a)- 2\psi(c-a+1)+\psi(c-a+2)\\
&=\frac{1}{1-a} + \frac{1}{1+c-d}-\frac{1}{1-d} - \frac{1}{1+c-a}+\frac{1}{(c-d)(c-d+1)}-\frac{1}{(c-a)(c-a+1)}\\
&=\frac{1}{1-a}-\frac{1}{1-d}+\frac{1}{c-d}-\frac{1}{c-a}=\frac{(a-d)(c-1)(1+c-a-d)}{(1-a)(1-d)(c-d)(c-a)}.
\end{align*}
Further, converting all $\phi$ functions to hypergeometric functions,  using the expression \eqref{eq:g1-value} for $g$ and considering
\begin{align*}
&\sum_{k=0}^{\infty}\frac{\Gamma(x_1+k+1)\Gamma(x_2+k)\Gamma(x'_2+k+1)\Gamma(x_3+k+1)\Gamma(x_4+k+1) z^{k+1}}{\Gamma(2+k)\Gamma(x_5+k+1) \Gamma(x_6+k+1)\Gamma(x_7+k+1) (k+1)!} \\
&= \sum_{n=1}^{\infty} \frac{\Gamma(x_1+n)\Gamma^2(x_2+n-1)\Gamma(x'_2+n)\Gamma(x_3+n)\Gamma(x_4+n) z^n}{\Gamma(n+1)\Gamma(x_5+n) \Gamma(x_6+n)\Gamma(x_7+n) n!}\\
&=\frac{\Gamma(x_1)\Gamma(x_2-1)\Gamma(x'_2)\Gamma(x_3)\Gamma(x_4)}{\Gamma(x_5)\Gamma(x_6)\Gamma(x_7)} \bigg[ F\left(\begin{array}{c} 
 x_1,x_2-1,x'_2,x_3,x_4\\ 1, x_5,x_6,x_7
\end{array} \middle| z\right) -1\bigg]
\end{align*}
for $x_1 \in \{c-a,c-a+1 \}$, $x_2=c$, $x'_2=c-1$ $x_3 \in \{d,d-1 \}$, $x_4= \bfe$, $x_5 \in \{a,a-1 \}$, $x_6 \in \{c-d,c-d+1 \}$ and $x_7=\bff$, we get
\begin{align*}
& a(c-d)(1-d)(1+c-a)F\!\left(\!\!\begin{array}{c}
 c-a+2,1+c,c,d,\bfe+1\\ a,2,c-d+2,\bff+1
\end{array} \middle| z\!\right)\\
&\times\bigg[F\!\left(\!\!\begin{array}{c}
 c-a,c-1,c-1,d,\bfe\\ 1,a,c-d,\bff
\end{array} \middle| z\!\right)-1 \bigg]-d(c-a)(1-a)(1+c-d) \\
&\times F\!\left(\!\!\begin{array}{c}
 c-a+1,c,c+1,d+1,\bfe+1\\ a+1,2,c-d+1,\bff+1
\end{array} \middle| z\!\right)\bigg[F\!\left(\!\!\begin{array}{c}
 c-a+1,c-1,c-1,d-1,\bfe\\ 1,a-1,c-d+1,\bff
\end{array} \middle| z\!\right)-1 \bigg]\\
&=c(a-d)(1+c-a-d) F\!\left(\!\!\begin{array}{c}
 1+c-a,c-1,c-1,d,\bfe\\ a-1,1,c-d,\bff
\end{array} \middle| z\!\!\right)\\
&\times\! F\!\left(\!\!\begin{array}{c}
 1+c-a,1+c,c+1,d,\bfe+1\\ a+1,1,c-d+2,\bff+1
\end{array} \middle| z\!\right)+z \frac{cd(a\!-\!d)(1\!+\!c\!-\!a)(1\!+\!c\!-\!a\!-\!d)(c\!-\!1)^2(\bfe)_1}{(c-d)(1-a)(\bff)_1}  \notag\\
&  \times F\!\left(\!\!\begin{array}{c}
 c-a+2,1+c,c,d,\bfe+1\\ a,2,c-d+2,\bff+1
\end{array} \middle| z\right) F\!\left(\!\!\begin{array}{c} 
 c-a+1,c,c+1,d+1,\bfe+1\\ a+1,2,c-d+1,\bff+1
\end{array} \middle| z\right) \notag\\
&-a(c-d)(1-d)(1+c-a) F\!\left(\!\!\begin{array}{c}
 c-a+2,1+c,c,d,\bfe+1\\ a,2,c-d+2,\bff+1
\end{array} \middle| z\right) \notag\\
&+d(c-a)(1-a)(1+c-d) F\!\left(\!\!\begin{array}{c}
 1+c-a,1+c,c,d+1,\bfe+1\\ a+1,2,1+c-d,\bff+1
\end{array} \middle| z\right),
\end{align*} 
Now canceling identical terms on the left-hand and right-hand sides
of the identity, we arrive at the conclusion of Theorem \ref{th:contiguous-degenerate-identity}. \end{proof}

\begin{remark}
Note that while differentiating 
\begin{multline*}
\phi\left(\begin{array}{c}
 \epsilon+c-a,\epsilon+c-1,c,d,\bfe\\ a,\epsilon,\epsilon+c-d,\bff
\end{array} \middle| z\right) \\= \sum_{k=0}^{\infty} \frac{\Gamma(\epsilon+c-a+k)\Gamma(\epsilon+c-1+k)\Gamma(c+k)\Gamma(d+k)\Gamma(\bfe+k) z^k}{\Gamma(a+k)\Gamma(\epsilon+k)\Gamma(\epsilon+c-d+k)\Gamma(\bff+k)k!}
\end{multline*}
with respect to $\epsilon$ in the previous theorem, one has to deal with the case when $k=0$ and $k \geq 1$ separately since when $k=0$, the first term contains the factor $\frac{\Gamma(c-a+\epsilon)\Gamma(c-1+\epsilon)}{\Gamma(\epsilon)\Gamma(c-d+\epsilon)}$ which vanishes at $\epsilon=0$ so that its logarithmic derivative is singular. Its derivative with respect to $\epsilon$ evaluated at $\epsilon=0$ is calculated as follows. Let
\begin{align}\label{eq:t-epsilon}
    t(\epsilon)=\frac{\Gamma(c-a+\epsilon)\Gamma(c-1+\epsilon)}{\Gamma(\epsilon)\Gamma(c-d+\epsilon)}.
\end{align}
The behavior of individual gamma functions involved in the above expression is given by
\begin{align*}
    \frac{1}{\Gamma(\epsilon)} = \epsilon+\gamma\epsilon^2+O(\epsilon^3), \quad \Gamma(T+\epsilon)= \Gamma(T)[1+\epsilon \psi(T)+O(\epsilon^2)],
\end{align*}
where $T$ is either $c-a$, $c-1$ or $c-d$, and $\gamma$ is the Euler's constant. Substituting in \eqref{eq:t-epsilon}, we have
\begin{align*}
 t(\epsilon)&= \frac{\Gamma(c-a)\Gamma(c-1)}{\Gamma(c-d)} \frac{[1+\epsilon\psi(c-a)] [1+\epsilon\psi(c-1)]}{[1+\epsilon\psi(c-d)]}[\epsilon+\gamma\epsilon^2+O(\epsilon^3)]\\
 &=\frac{\Gamma(c-a)\Gamma(c-1)}{\Gamma(c-d)} [\epsilon+(\gamma+\psi(c-a)+\psi(c-1)-\psi(c-d))\epsilon^2+O(\epsilon^3)].
\end{align*}
Differentiating and evaluating at $\epsilon=0$, we get $t'(0)=\frac{\Gamma(c-a)\Gamma(c-1)}{\Gamma(c-d)}$. A similar procedure needs to be followed while differentiating $$
\phi\left(\begin{array}{c}
 1+\epsilon+c-a,\epsilon+c-1,c,d-1,\bfe\\ a-1,\epsilon,1+\epsilon+c-d,\bff
\end{array} \middle| z\right)
$$ 
appearing in $h_2(\epsilon)$ of the previous theorem.
\end{remark}

\begin{proof}[Proof of Theorem~\ref{th:two-index-degenerate-identity}] Identity \eqref{eq:two-index-contiguous-identity} is valid  whenever $b_1 \neq 1$. Our goal is to compute the limit as $b_1\to1$. To this end, substituting $b_1=1+\epsilon$ in \eqref{eq:two-index-contiguous-identity} and using \eqref{eq:regularized-hypergeometric}, we obtain
\begin{align}\label{eq:two-index-identity-regularized}
&\frac{a_1(a_2-1-\epsilon)\Gamma(\epsilon) \Gamma(\bfb_{[1]}) \Gamma(\epsilon+2) \Gamma(\bfb_{[1]}+1)}{\Gamma(a_1-1)\Gamma(\bfa_{[1]}) \Gamma(a_2) \Gamma(\bfa_{[2]}+1)}  \phi\left(\begin{array}{c}
 a_1-1,\bfa_{[1]}\\ \epsilon,\bfb_{[1]}
\end{array} \middle| z\right) \phi\left(\begin{array}{c}
 a_2,\bfa_{[2]}+1\\ \epsilon+2, \bfb_{[1]}+1
\end{array} \middle| z\right) \notag\\
&-\frac{a_2(a_1-1-\epsilon)\Gamma(\epsilon) \Gamma(\bfb_{[1]}) \Gamma(\epsilon+2) \Gamma(\bfb_{[1]}+1)}{\Gamma(a_2-1)\Gamma(\bfa_{[2]}) \Gamma(a_1) \Gamma(\bfa_{[1]}+1)}  \phi\left(\begin{array}{c}
 a_2-1,\bfa_{[2]}\\ \epsilon,\bfb_{[1]}
\end{array} \middle| z\right) \phi\left(\begin{array}{c}
 a_1,\bfa_{[1]}+1\\ \epsilon+2, \bfb_{[1]}+1
\end{array} \middle| z\right)\notag\\
&=(1+\epsilon)(a_2-a_1) F\left(\begin{array}{c}
 \bfa\\ 1+\epsilon,\bfb_{[1]}
\end{array} \middle| z\right) F\left(\begin{array}{c}
 a_1,a_2,\bfa_{[1,2]}+1\\ 1+\epsilon,\bfb_{[1]}+1
\end{array} \middle| z\right).
\end{align}
Consider the functions
\begin{align*}
&f_1(\epsilon)=\Gamma(\epsilon) h_1(\epsilon)=\Gamma(\epsilon)g_1(\epsilon) \phi\left(\begin{array}{c}
 a_1-1,\bfa_{[1]}\\ \epsilon,\bfb_{[1]}
\end{array} \middle| z\right) \phi\left(\begin{array}{c}
 a_2,\bfa_{[2]}+1\\ \epsilon+2, \bfb_{[1]}+1
\end{array} \middle| z\right),\\
&f_2(\epsilon)=\Gamma(\epsilon) h_2(\epsilon)=\Gamma(\epsilon)g_2(\epsilon) \phi\left(\begin{array}{c}
 a_2-1,\bfa_{[2]}\\ \epsilon,\bfb_{[1]}
\end{array} \middle| z\right) \phi\left(\begin{array}{c}
 a_1,\bfa_{[1]}+1\\ \epsilon+2, \bfb_{[1]}+1
\end{array} \middle| z\right),
\end{align*}
where
\begin{align*}
&g_1(\epsilon) =\frac{a_1(a_2-1-\epsilon) \Gamma(\bfb_{[1]}) \Gamma(\epsilon+2) \Gamma(\bfb_{[1]}+1)}{\Gamma(a_1-1)\Gamma(\bfa_{[1]}) \Gamma(a_2) \Gamma(\bfa_{[2]}+1)},\\
&g_2(\epsilon) =\frac{a_2(a_1-1-\epsilon)\Gamma(\bfb_{[1]}) \Gamma(\epsilon+2) \Gamma(\bfb_{[1]}+1)}{\Gamma(a_2-1)\Gamma(\bfa_{[2]}) \Gamma(a_1) \Gamma(\bfa_{[1]}+1)}.
\end{align*} 
Expand  $h_1(\epsilon)$ and $h_2(\epsilon)$ as in \eqref{eq:contiguous-degeneration-limit}. Now, we calculate $h'_1(0)$ and $h'_2(0)$.
\begin{align*}
&h'_1(0) = \left[ \frac{\partial}{\partial\epsilon} h_1(\epsilon)  \right]_{\epsilon = 0}= \bigg(\psi(2)-\frac{1}{a_2-1} \bigg) g_1(0) \phi\left(\begin{array}{c}
 a_1-1,\bfa_{[1]}\\ 0,\bfb_{[1]}
\end{array} \middle| z\right)\\
& \phi\left(\begin{array}{c}
 a_2,\bfa_{[2]}+1\\ 2, \bfb_{[1]}+1
\end{array} \middle| z\right)+g_1(0)\bigg[\frac{\Gamma(a_1-1)\Gamma(\bfa_{[1]})}{\Gamma(\bfb_{[1]})} -\sum_{k=1}^{\infty} \frac{\Gamma(a_1-1+k)\Gamma(\bfa_{[1]}+k) \psi(k)z^k}{\Gamma(k)\Gamma(\bfb_{[1]}+k)k!} \bigg]\\
& \phi\left(\begin{array}{c}
 a_2,\bfa_{[2]}+1\\ 2, \bfb_{[1]}+1
\end{array} \middle| z\right)-g_1(0) \phi\left(\begin{array}{c}
 a_1-1,\bfa_{[1]}\\ 0,\bfb_{[1]}
\end{array} \middle| z\right)\sum_{k=0}^{\infty}\frac{\Gamma(a_2+k)\Gamma(\bfa_{[2]}+1+k) \psi(2+k)z^k}{\Gamma(2+k)\Gamma(\bfb_{[1]}+1+k)k!}.
\end{align*}
Similarly,
\begin{align*}
&h'_2(0) = \left[ \frac{\partial}{\partial\epsilon} h_2(\epsilon)  \right]_{\epsilon = 0}= \bigg(\psi(2)-\frac{1}{a_1-1} \bigg) g_2(0) \phi\left(\begin{array}{c}
 a_2-1,\bfa_{[2]}\\ 0,\bfb_{[1]}
\end{array} \middle| z\right)\\
& \phi\left(\begin{array}{c}
 a_1,\bfa_{[1]}+1\\ 2, \bfb_{[1]}+1
\end{array} \middle| z\right)+g_2(0) \bigg[\frac{\Gamma(a_2-1)\Gamma(\bfa_{[2]})}{\Gamma(\bfb_{[1]})} -\sum_{k=1}^{\infty} \frac{\Gamma(a_2-1+k)\Gamma(\bfa_{[2]}+k) \psi(k)z^k}{\Gamma(k)\Gamma(\bfb_{[1]}+k)k!} \bigg]\\
& \phi\left(\begin{array}{c}
 a_1,\bfa_{[1]}+1\\ 2, \bfb_{[1]}+1
\end{array} \middle| z\right)-g_2(0) \phi\left(\begin{array}{c}
 a_2-1,\bfa_{[2]}\\ 0,\bfb_{[1]}
\end{array} \middle| z\right)\sum_{k=0}^{\infty}\frac{\Gamma(a_1+k)\Gamma(\bfa_{[1]}+1+k) \psi(2+k)z^k}{\Gamma(2+k)\Gamma(\bfb_{[1]}+1+k)k!}.
\end{align*}
Clearly, $g_1(0)=g_2(0)$ which will be denoted by $g$, i.e.,
\begin{align*}
g_1(0)=g_2(0)=\frac{a_1a_2(a_1-1)(a_2-1) \Gamma(\bfb_{[1]}) \Gamma(\bfb_{[1]}+1)}{\Gamma(\bfa)  \Gamma(\bfa+1)}=:g.
\end{align*} 
Substituting these derivatives into the limit formula \eqref{eq:contiguous-degeneration-limit}
and then into \eqref{eq:two-index-identity-regularized} gives a considerably shorter
expression after two elementary simplifications. Namely,
\[
\phi\left(\begin{array}{c}a_i-1,\mathbf a_{[i]}\\0,\mathbf b_{[1]}\end{array}\middle|z\right)
=z\phi\left(\begin{array}{c}a_i,\mathbf a_{[i]}+1\\2,\mathbf b_{[1]}+1\end{array}\middle|z\right)
\qquad(i=1,2),
\]
and the sums beginning at \(k=1\) are shifted to begin at \(k=0\).
Pairing the shifted sums and using
\(\psi(k+2)-\psi(k+1)=1/(k+1)\), we obtain
\begin{align*}
&g\bigg[\phi\left(\begin{array}{c}
 a_2,\bfa_{[2]}+1\\ 2, \bfb_{[1]}+1
\end{array} \middle| z\right)\sum_{k=0}^{\infty}\frac{\Gamma(a_1+k)\Gamma(\bfa_{[1]}+1+k) z^{k+1}}{\Gamma(2+k)\Gamma(\bfb_{[1]}+1+k) (k+1)!} \\
&\qquad-\phi\left(\begin{array}{c}
 a_1,\bfa_{[1]}+1\\ 2, \bfb_{[1]}+1
\end{array} \middle| z\right)\sum_{k=0}^{\infty}\frac{\Gamma(a_2+k)\Gamma(\bfa_{[2]}+1+k) z^{k+1}}{\Gamma(2+k)\Gamma(\bfb_{[1]}+1+k) (k+1)!}\bigg]\\
&=(a_2-a_1) F\left(\begin{array}{c}
 \bfa\\ 1,\bfb_{[1]}
\end{array} \middle| z\right) F\left(\begin{array}{c}
 a_1,a_2,\bfa_{[1,2]}+1\\ 1,\bfb_{[1]}+1
\end{array} \middle| z\right) \\
&\quad+g\frac{a_1-a_2}{(a_1-1)(a_2-1)}z
 \phi\left(\begin{array}{c}a_2,\bfa_{[2]}+1\\2,\bfb_{[1]}+1\end{array}\middle|z\right)
 \phi\left(\begin{array}{c}a_1,\bfa_{[1]}+1\\2,\bfb_{[1]}+1\end{array}\middle|z\right)\\
&\quad+a_2(a_1-1)F\left(\begin{array}{c}
 a_1,\bfa_{[1]}+1\\2,\bfb_{[1]}+1
\end{array}\middle|z\right)
-a_1(a_2-1)F\left(\begin{array}{c}
 a_2,\bfa_{[2]}+1\\2,\bfb_{[1]}+1
\end{array}\middle|z\right).
\end{align*}
Further, converting all $\phi$ functions to hypergeometric using the expression of $g$ and considering
\begin{align*}
&\sum_{k=0}^{\infty}\frac{\Gamma(x+k)\Gamma(y+1+k) z^{k+1}}{\Gamma(2+k)\Gamma(w+1+k) (k+1)!} = \sum_{n=1}^{\infty} \frac{\Gamma(x+n-1)\Gamma(y+n) z^{n}}{\Gamma(n+1)\Gamma(w+n) n!}\\
&=\frac{\Gamma(x-1)\Gamma(y)}{\Gamma(w)} \bigg[ F\left(\begin{array}{c} 
 x-1,y\\ 1, w
\end{array} \middle| z\right) -1\bigg]
\end{align*}
for $x \in \{a_1,a_2 \}$, $y \in \{\bfa_{[1]},\bfa_{[2]} \}$, and $w=\bfb_{[1]}$, we obtain
\begin{align*}
&a_1(a_2-1) F\left(\begin{array}{c}
 a_2,\bfa_{[2]}+1\\ 2, \bfb_{[1]}+1
\end{array} \middle| z\right)\bigg[ F\left(\begin{array}{c} 
 a_1-1,\bfa_{[1]}\\ 1, \bfb_{[1]}
\end{array} \middle| z\right) -1\bigg]  \notag\\
& -a_2(a_1-1) F\left(\begin{array}{c}
 a_1,\bfa_{[1]}+1\\ 2, \bfb_{[1]}+1
\end{array} \middle| z\right) \bigg[ F\left(\begin{array}{c} 
 a_2-1,\bfa_{[2]}\\ 1, \bfb_{[1]}
\end{array} \middle| z\right) -1\bigg]\notag\\
&=(a_2-a_1) F\left(\begin{array}{c}
 \bfa\\ 1,\bfb_{[1]}
\end{array} \middle| z\right) F\left(\begin{array}{c}
 a_1,a_2,\bfa_{[1,2]}+1\\ 1,\bfb_{[1]}+1
\end{array} \middle| z\right) \notag\\
&+z\frac{(\bfa)_1(a_1-a_2)}{(\bfb_{[1]})_1}    F\left(\begin{array}{c} 
 a_2,\bfa_{[2]}+1\\2, \bfb_{[1]}+1
\end{array} \middle| z\right)  F\left(\begin{array}{c}
 a_1,\bfa_{[1]}+1\\ 2, \bfb_{[1]}+1
\end{array} \middle| z\right) \notag\\
&+a_2(a_1-1) F\left(\begin{array}{c}
 a_1,\bfa_{[1]}+1\\ 2, \bfb_{[1]}+1
\end{array} \middle| z\right)-a_1(a_2-1) F\left(\begin{array}{c}
 a_2,\bfa_{[2]}+1\\ 2, \bfb_{[1]}+1
\end{array} \middle| z\right).
\end{align*}
Canceling the identical single-function terms on the two sides gives
\eqref{eq:two-index-degenerate-identity}, and hence proves Theorem~\ref{th:two-index-degenerate-identity}. \end{proof}

\end{document}